\documentclass{article}

\usepackage[T2C]{fontenc}
\usepackage[utf8]{inputenc}
\usepackage[english,russian]{babel}
\usepackage{amsfonts,amssymb,mathrsfs,amscd,longtable}
\usepackage{mathrsfs}
\usepackage[umnobzor]{sks2e-tru}
\usepackage[all]{xy}
\usepackage[most]{tcolorbox}
\usepackage{microtype}
\usepackage{booktabs} 
\usepackage{mathtools}
\mathtoolsset{showonlyrefs}

\usepackage{algorithm}
\usepackage{algorithmic}

\usepackage{dsfont}

\usepackage{hyperref}
\usepackage{xurl}

\usepackage[round]{natbib}
\renewcommand{\bibname}{References}

\usepackage{makecell}
\usepackage{pifont}

\def\bad{\spaceskip=0.33emplus0.6emminus0.15em\immediate\write5{\string\bad}}

\numberwithin{equation}{section}

\makeatletter

\def\PlotAlph#1{\expandafter\@PlotAlph\csname c@#1\endcsname}

\def\@PlotAlph#1{%
 \ifcase#1\or $\mathrm{A}1$\else \@ctrerr \fi
}%

\theoremstyle{plain}
\newtheorem{theorem}{Theorem}

\newtheorem{lemma}{Lemma}

\theoremstyle{definition}

\newtheorem{proof}{Proof}
\newtheorem{remark}{Remark}
\newtheorem{example}{Example}

\renewcommand{\epsilon}{\varepsilon}

\allowdisplaybreaks

\usepackage{url}
\usepackage{algorithm}
\usepackage{algorithmic}
\usepackage{tabularx}
\usepackage{paralist}
\usepackage{mathtools}

\usepackage{bbm} 

\usepackage{makecell}
\usepackage{multirow}
\usepackage{booktabs}

\usepackage{nicefrac}       

\usepackage[flushleft]{threeparttable} 

\usepackage{multirow}
\usepackage{colortbl}
\definecolor{bgcolor}{rgb}{0.8,1,1}
\definecolor{bgcolor2}{rgb}{0.8,1,0.8}
\definecolor{niceblue}{rgb}{0.0,0.19,0.56}

\usepackage{hyperref}
\usepackage{sidecap}
\usepackage{pifont}

\usepackage{subfigure}

\newcommand{\R}{\mathbb{R}}

\def\<#1,#2>{\left\langle #1,#2\right\rangle}

\usepackage{xspace}

\newcommand{\argmin}{\mathop{\arg\!\min}}

\newcommand{\cD}{{\cal D}}

\newcommand{\cO}{{\cal O}}

\newcommand{\EE}{\mathbb{E}}

\makeatletter
\def\@@Head@Line{\vbox{\parindent0pt\large\bfseries\itshape
\begin{center}\selectlanguage{russian}\the\JOURNAL@TOKS
\vspace{1.8mm}\hrule height 0mm depth 1pt\end{center}}}
\def\@@copyright#1{}
\makeatother
\def\KeyWordsText{Keywords}
\hypersetup{hidelinks,pdftitle={Stochastic Optimization and Data Science},
pdfauthor={A. Avetisyan, D. Dvinskikh, A. Gasnikov, N. Tupitsa, D. Turdakov}}

\begin{document}

\title{Stochastic Optimization and Data Science}
\author[A.~Avetisyan et al.]{A.~Avetisyan, D.~Dvinskikh, A.~Gasnikov, N.~Tupitsa, D.~Turdakov}
\address{Университет Иннополис, Высшая школа экономики, Институт системного программирования РАН им. В.П. Иванникова, Математический институт РАН им. В.А. Стеклова, Сколтех, МФТИ}


\date{27.03.2026}
\udk{517.538}

\maketitle
\selectlanguage{english}

\begin{fulltext}
\frenchspacing

\begin{abstract}
This survey explains how stochastic optimization problems arise in mathematical statistics and statistical learning, through maximum likelihood estimation and population risk minimization. We compare two main approaches to expectation minimization: offline Sample Average Approximation and online Stochastic Approximation. Particular attention is paid to the distinction between statistical accuracy and optimization accuracy, and to the roles of convexity, regularization, stability, and growth conditions. We then discuss how additional structure changes the computational picture: heavy-tailed noise, adaptive and accelerated methods, variance reduction, distributed optimization, and higher-order information. These developments connect classical sample-complexity bounds with current stochastic and finite-sum algorithms.
\end{abstract}

\begin{keywords}
stochastic optimization, statistical learning, sample average approximation, stochastic approximation, empirical risk minimization, population risk minimization, maximum likelihood estimation, stochastic gradient methods, sample complexity, high-probability guarantees, Bayesian estimation, regularization
\end{keywords}

\markright{Stochastic optimization and Data science}

\footnotetext[0]{ This work was supported by a grant, provided by the Ministry of Economic Development of the Russian Federation (agreement  dated June 20, 2025 No. 139-15-2025-011, identifier 000000C313925P4G0002).
}

\section{Introduction}
This survey is devoted to stochastic optimization problems and their role in modern data analysis. Many statistical estimation and learning procedures can be understood as methods for minimizing an expected loss:
\begin{equation}\label{eq:SOP}
     \min_{x\in Q\subseteq\R^n}\left\{f(x):=\EE_{\xi\sim\cD}[f(x,\xi)]\right\}.
\end{equation}
Here $\cD$ is an unknown distribution, $f(x,\xi)$ is a sample loss, and $f(x)$ is the population objective. The central case is convex optimization, although several comparisons also concern nonconvex sample losses. Convexity of each loss and convexity of its expectation will be distinguished when this matters.

The first part of the survey explains how these problems arise in mathematical statistics and machine learning. We then compare two approaches: \textit{offline} Monte Carlo or Sample Average Approximation (SAA), which optimizes an empirical objective on a fixed sample, and \textit{online} Stochastic Approximation (SA), which processes fresh observations sequentially. Fundamental results for SA go back to \cite{nemirovski1983problem}; see also \cite{nemirovski2009robust}. The statistical theory of SAA, regularization, and empirical risk minimization was developed in particular in \cite{shapiro2005complexity,shalev2009stochastic}, with a sharp characterization of general convex ERM sample complexity in \cite{pmlr-v238-carmon24a}. The concluding part describes how additional assumptions improve computational efficiency while respecting the statistical limitations of the underlying problem.

Throughout the nonasymptotic discussion, $Q$ is nonempty, closed, and convex, samples are independent and identically distributed, and the losses and algorithmic outputs are measurable. Risks are finite; under a Lipschitz assumption it suffices to assume integrability at one fixed point. Minima used in the statements are assumed to exist, with measurable choices of empirical minimizers. We write $x^*\in\argmin_{x\in Q}f(x)$ and $f^*=f(x^*)=\min_{x\in Q}f(x)$. In nonsmooth problems, the notation $\nabla_x f(x,\xi)$ denotes a chosen stochastic subgradient satisfying the stated norm or moment bound and whose conditional expectation belongs to $\partial f(x)$. An online iterate depends only on past observations; the next oracle sample is independent of this past. Expectations of algorithmic errors include both sampling and internal randomization. The notation $\widetilde{\cO}$ suppresses logarithmic factors; a sample bound, an oracle bound, and an arithmetic-operation bound need not be the same.

We use $\langle x,y\rangle=\sum_{i=1}^n x_i y_i$ for the Euclidean inner product and $I_n$ for the $n\times n$ identity matrix. The notation $\|x\|_p$ denotes the usual $p$-norm. Its conjugate exponent is denoted throughout by $q$, with $1/p+1/q=1$ and the convention $1/\infty=0$; thus the dual norm is $\|\cdot\|_q$. We write
\[
B_p^n(R)=\{x\in\R^n:\|x\|_p\leq R\},\qquad
B_p(x_0,R)=\{x\in\R^n:\|x-x_0\|_p\leq R\}.
\]
Here $R\geq0$ and $x_0\in\R^n$. The symbols $\EE$ and $\mathds P$ denote expectation and probability. For a convex function, $\partial f(x)$ denotes its subdifferential at $x$. Constants denoted by $C$ may change between inequalities; any dependence beyond universal numerical constants is stated explicitly.

\section{Motivation to stochastic optimization}\label{Sec:StocMotiv}
Stochastic optimization problems arise whenever decisions must account for uncertain data, from telecommunications and medicine to finance \cite{shapiro2021lectures}. Their analysis brings together optimization, probability, statistics, and functional analysis. The growing role of data science has strengthened these connections and stimulated both theory and algorithms; see \cite{shapiro2021lectures}, \cite{shalev2014understanding}, \cite{bach2021learning}.

In this section, we briefly describe two main origins of stochastic optimization problems in data science: the \textit{statistical} approach, illustrated by maximum likelihood estimation, and the \textit{machine learning} approach, based on expected risk minimization and stochastic gradient descent.

\subsection{Statistical motivation}\label{Sec:StatisticsMotiv}
We start with the 
simplest example.
Let $x^\star \in \R$ be an unknown scalar parameter,  $\eta \sim \mathcal{N}\left(0,\sigma^2\right)$ be Gaussian noise. Assume that we can measure
\begin{center}
$\xi^k = x^\star + \eta^k$, $k=1,...,N$,
\end{center}
where $\eta^k$ are i.i.d. (independent identically distributed as $\eta$). \textbf{The goal is to estimate $x^\star$ from $\left\{\xi^k\right\}_{k=1}^N$.}

The main observation is the following: $x^\star$ is a solution of the stochastic optimization problem
\begin{equation}
    \min\limits_{x\in \R} \EE_{\xi}\left[f(x,\xi)\right], \label{eq:exp_minimization_2}
\end{equation}
where $f(x,\xi)=(\xi-x)^2$ and $\xi \sim \mathcal{N}\left(x^\star,\sigma^2\right)$. Indeed, 
\[
\EE_{\xi}(\xi-x)^2
=\EE_{\xi}\xi^2-2x\EE_{\xi}\xi+x^2
=(x^\star-x)^2+\sigma^2
\]
attains its minimum at $x=x^\star$. 
However, $x^\star$ is unknown, and $\sigma^2$ may also be unknown. How can problem \eqref{eq:exp_minimization_2} be solved?
Since $\left\{\xi^k\right\}_{k=1}^N$ 
are available, 
the Monte Carlo approach can be employed. This approach consists in  replacing  problem \eqref{eq:exp_minimization_2} by its empirical version
\begin{equation}
    \min\limits_{x\in \R} \frac{1}{N}\sum_{k=1}^N (\xi^k - x)^2. \label{eq:exp_minimization_3}
\end{equation}
The solution of \eqref{eq:exp_minimization_3} is
\begin{equation}
\bar{x}^N = \frac{1}{N}\sum_{k=1}^N \xi^k.
\label{eq:bar_x}
\end{equation}
In statistics, this  average is known  as the 
{\textit{sample mean},}
which is unbiased and has the smallest variance among unbiased estimators in this Gaussian model (see Theorem~\ref{Th:leCam1}).

The same solution can also be obtained by the following online procedure, starting from any $x^0$:
\begin{equation}
    x^k=x^{k-1}-\frac{1}{2k}\nabla_x f(x^{k-1},\xi^k)
    =\left(1-\frac{1}{k}\right)x^{k-1}+\frac{1}{k}\xi^k,
    \qquad k=1,\ldots,N.
\end{equation}
Thus $x^N=\bar{x}^N$, where $\nabla_x f(x,\xi)$ is the gradient of $f$ with respect to $x$.
This procedure is \textit{stochastic gradient descent} (SGD) for the $2$-strongly convex stochastic optimization problem \eqref{eq:exp_minimization_2}.

A natural question is how to choose the loss $f(x,\xi)$. There may be many functions with the same expected minimizer. Maximum likelihood offers a principled choice when a statistical model is available.

Assume that a random observation $\xi$ depends on an unknown parameter $x^\star\in\Theta$, where $\Theta\subseteq\R^n$ is the parameter space. Let $p(x,\xi)$ denote its probability mass function or density under parameter $x\in\Theta$. In the 
aforementioned
example, $n=1$ and probability density function was
$$p(x,\xi) = \frac{1}{\sqrt{2\pi\sigma^2}}\exp\left\{-\frac{(\xi - x)^2}{2\sigma^2}\right\}.$$
If i.i.d. samples $\left\{\xi^k\right\}_{k=1}^N$ are available let us introduce the likelihood $$p\left(x,\left\{\xi^k\right\}_{k=1}^N\right) = \prod_{k=1}^N p(x,\xi^k).$$
Perhaps, one of the most productive ideas in statistics is to estimate the true vector of parameters $x^\star$ as a vector that maximize{s} likelihood $p\left(x,\left\{\xi^k\right\}_{k=1}^N\right)$. This problem can be equivalently reformulated as minimization of (normalized) negative log-likelihood
\begin{equation}\label{eq:MLE_N}
    \min\limits_{x\in\Theta}\left[-\frac{1}{N}\sum_{k=1}^N\log p(x,\xi^k)\right].
\end{equation}
This minimization problem can be considered as the empirical (sometimes called \textit{Monte Carlo}) version of the stochastic optimization problem 
\begin{equation}
\min\limits_{x\in\Theta}\EE_{\xi}\left[-\log p(x,\xi)\right].
\label{eq:x_star}
\end{equation}
In particular, for
the aforementioned
Gaussian model, problem \eqref{eq:x_star} has the following form  
$$\min\limits_{x\in \R} \EE_{\xi} \left[\frac{1}{2\sigma^2}(\xi - x)^2 + \frac{1}{2}\log\left(2\pi\sigma^2\right)\right],$$
which is equivalent to \eqref{eq:exp_minimization_2}.

More generally, suppose that the model is correctly specified: the observations have density $p(x^\star,\cdot)$. We write $\EE_{x^\star}$ for expectation under this distribution. The expected log-likelihood ratio satisfies
\[
\EE_{x^\star}\log\frac{p(x^\star,\xi)}{p(x,\xi)}
=KL\left(p(x^\star,\cdot),p(x,\cdot)\right)\ge0.
\]
Here $KL$ is the Kullback--Leibler divergence; the value $+\infty$ is allowed. If the expected negative log-likelihoods are finite, this identity implies
\[
x^\star\in\operatorname{Arg}\min_{x\in\Theta}\EE_{x^\star}[-\log p(x,\xi)].
\]
The minimizer is unique when the model is identifiable, meaning that different parameters give different distributions. Thus maximum likelihood is the empirical counterpart of stochastic optimization with loss $f(x,\xi)=-\log p(x,\xi)$. This statistical interpretation does not itself imply convexity; the convex methods discussed below apply to those models and losses for which convexity holds.

A central result of classical statistics describes the asymptotic behavior of the \textit{maximum likelihood estimator} (MLE):
\begin{equation}\label{eq:MLE}
\begin{aligned}
\hat{x}_{\mathrm{MLE}}^N
&\in\operatorname{Arg}\max_{x\in\Theta}\prod_{k=1}^Np(x,\xi^k)\\
&=\operatorname{Arg}\min_{x\in\Theta}
\left[-\sum_{k=1}^N\log p(x,\xi^k)\right].
\end{aligned}
\end{equation}
We state a finite-dimensional regular version, keeping the assumptions explicit but without developing the full asymptotic theory; see \cite{vandervaart1998asymptotic}, \cite{ibragimov2013statistical}.
\begin{theorem}\label{Th:leCam1}
Let the model be identifiable, with common support and an interior true parameter $x^\star$. Assume that the log-density is twice continuously differentiable near $x^\star$, differentiation under the integral is valid, and the Fisher information matrix
\[
I_{x^\star}=\EE_{x^\star}\left[\nabla_x\log p(x^\star,\xi)
\nabla_x\log p(x^\star,\xi)^T\right]
\]
is finite and positive definite. Then:
\begin{enumerate}
\item Every estimator $\tilde x^N$ that is unbiased in a neighborhood of $x^\star$, has finite second moments, and whose expectation can be differentiated satisfies the \textit{Cram\'er--Rao inequality}
\[
\EE_{x^\star}\left[(\tilde x^N-x^\star)(\tilde x^N-x^\star)^T\right]
\succcurlyeq (NI_{x^\star})^{-1}.
\]
Here $A\succcurlyeq B$ means $\langle z,(A-B)z\rangle\ge0$ for every $z$.
\item If the selected MLE is consistent and the observed information per observation converges locally uniformly in probability to its expectation, continuous at $x^\star$, then
\[
\sqrt{N}(\hat x_{\mathrm{MLE}}^N-x^\star)
\ \Longrightarrow\ \mathcal N(0,I_{x^\star}^{-1}),
\qquad N\to\infty.
\]
\end{enumerate}
\end{theorem}
Here $\Longrightarrow$ denotes \textit{convergence in distribution}.

The second statement follows by expanding the likelihood equation around $x^\star$ and applying the central limit theorem. It describes the limiting distribution, not an exact finite-sample Gaussian law. Convergence of the scaled covariance additionally requires uniform integrability of the squared scaled errors. The resulting efficiency bound concerns regular estimators; it is not a pointwise variance comparison with every biased estimator. The James--Stein estimator illustrates why this distinction matters \cite{james1961estimation}.

Under these regularity conditions, the MLE and a consistent estimate of its information matrix yield asymptotically efficient confidence regions. An online method can attain the same limiting covariance through \textit{Polyak--Juditsky--Ruppert averaging} \cite{ruppert1988efficient}, \cite{polyak1992acceleration}. Consider a smooth expected loss with an interior minimizer and an unbiased stochastic gradient $g(x,\xi)$:
\begin{equation}\label{eq:PRaveraging}
x^k=x^{k-1}-\gamma_k g(x^{k-1},\xi^k),
\qquad
\bar x_{\mathrm{PR}}^N=\frac{1}{N}\sum_{k=0}^{N-1}x^k.
\end{equation}
For decreasing stepsizes $\gamma_k=c(k+k_0)^{-a}$, with $c>0$, $k_0\ge0$, and $1/2<a<1$, the classical averaging theorem gives, under its local smoothness, stability, and noise moment conditions,
\[
\sqrt{N}(\bar x_{\mathrm{PR}}^N-x^\star)
\Longrightarrow\mathcal N(0,\Sigma_{\mathrm{PR}}),
\qquad N\to\infty,
\]
where
\[
\Sigma_{\mathrm{PR}}=H^{-1}SH^{-1},\qquad
H=\nabla^2 f(x^\star)\succ0,\qquad
S=\EE\bigl[g(x^\star,\xi)g(x^\star,\xi)^T\bigr].
\]
For the correctly specified likelihood loss, with $g(x,\xi)=-\nabla_x\log p(x,\xi)$, the information identity gives $H=S=I_{x^\star}$. Thus averaging can recover the MLE covariance without inverting a Hessian at each iteration.

Recent work quantifies the accuracy of this central limit theorem for \emph{nonlinear} SGD. The relevant distance is
\[
d_{\mathrm{conv}}(Y,Z)
=\sup_{C\in\mathcal C_n}
\left|\mathbb P(Y\in C)-\mathbb P(Z\in C)\right|,
\]
where $\mathcal C_n$ consists of convex Borel subsets of $\R^n$; this distance controls the coverage error of convex confidence regions. For a strongly convex smooth $f$, locally Lipschitz Hessian, nondegenerate $S$, and gradient noise with finite fourth moments and Lipschitz dependence on $x$ in fourth mean, \citet[Theorem~4]{sheshukova2026gaussian} give, for suitable $c,k_0$,
\[
d_{\mathrm{conv}}\!\left(\sqrt N(\bar x_{\mathrm{PR}}^N-x^\star),
\mathcal N(0,\Sigma_{\mathrm{PR}})\right)
\le C\left(N^{-1/2}+N^{-(a-1/2)}+N^{-(1-a)}\right).
\]
Here $C$ is independent of $N$, with dimension, problem parameters, and initialization fixed. Balancing the last two terms gives $a=3/4$ and error $\cO(N^{-1/4})$. A quadratic example shows that the term $N^{-(1-a)}$ cannot generally be removed, so this bound is sharp for $a\ge3/4$.

Finer mean squared error bounds describe how quickly the efficient leading term becomes dominant. For nonlinear strongly convex smooth problems, the restarted \emph{ROOT-SGD} method uses recursive gradient averaging and achieves \cite[Corollary~3]{li2022rootsgd}
\[
\EE\|x_{\mathrm{ROOT}}^N-x^\star\|_2^2
\le\frac{\operatorname{Tr}\Sigma_{\mathrm{PR}}}{N}
+\widetilde{\cO}(N^{-3/2}).
\]
Here $\widetilde{\cO}$ suppresses logarithmic factors in $N$. The bound requires a Lipschitz Hessian, fourth-moment noise bounds, Lipschitz noise in fourth mean, and the prescribed stepsize and restart schedule. The leading coefficient is one; the relative remainder decreases as $\widetilde{\cO}(N^{-1/2})$. ROOT-SGD modifies the gradient estimator, so this guarantee is not a bound for the simple average \eqref{eq:PRaveraging}.

Another approach combines two coupled, tail-averaged SGD runs with constant stepsizes $\gamma$ and $2\gamma$ by Richardson--Romberg extrapolation \cite{sheshukova2025richardson}. With $\gamma\asymp N^{-1/2}$, bounded third and fourth derivatives, suitable sixth-moment and stability bounds, and fixed initialization, it gives
\[
\left(\EE\|H(x_{\mathrm{RR}}^N-x^\star)\|_2^2\right)^{1/2}
\le\frac{\sqrt{\operatorname{Tr}S}}{\sqrt N}
+\cO(N^{-3/4}),
\]
where $x_{\mathrm{RR}}^N$ is twice the average at step $\gamma$ minus the average at step $2\gamma$, each over $N$ iterations after $N$ burn-in iterations. The exponentially decaying initialization contribution is absorbed in the remainder. This remainder concerns root mean squared error, a different quantity from normal approximation error.

In the linear case, refined analysis gives normal approximation error $\widetilde{\cO}(N^{-1/3})$ in convex distance with stepsize exponent $a=2/3$, under boundedness, stability, and nondegeneracy assumptions \cite{butyrin2025improved}. With a fixed positive stepsize, multiplicative noise can instead add a stepsize-dependent correction to the limiting covariance \cite{mou2020linear}.

Asymptotic theory does not fully describe a small sample. Consider the Bernoulli model
\[
p(x,\xi)=x^\xi(1-x)^{1-\xi},
\qquad x\in[0,1],\quad\xi\in\{0,1\},
\]
with $0<x^\star<1$ and the usual continuous extension at $x=0,1$. With probability $(1-x^\star)^N$, all observations vanish and $\hat x_{\mathrm{MLE}}^N=0$. This probability is appreciable when $Nx^\star$ is small, and the expected log-loss at the zero prediction is $+\infty$. Smoothing or regularization can therefore be important even in this elementary example.

Finite-sample and misspecified versions of likelihood theory are developed, for example, in \cite{spokoiny2012parametric}. Under misspecification, the target is a KL-projection of the true distribution onto the model, when this projection exists. The covariance of the Gaussian approximation generally has the sandwich form
\[
\frac{1}{N}H^{-1}JH^{-1},
\]
where $H$ is the expected negative log-likelihood Hessian and $J$ is the score covariance at that target; it need not equal inverse Fisher information. In this survey we mainly concentrate on non-asymptotic optimization methods, for which the sample size and computational accuracy are explicit.

At the end of this subsection, we explain the role of regularization as a Bayesian prior. Suppose that a parameter $X$ is first drawn from a prior density $\pi$, and the observations are then generated conditionally independently from $p(X,\cdot)$.

For quadratic loss, the Bayesian estimator is the posterior mean, provided the posterior has a finite second moment:
\begin{equation}\label{eq:bayesian_estimator}
\begin{split}
\hat x_{\mathrm B}^N
&=\operatorname{arg}\min_{x\in\R^n}\int_{\Theta}\|x-z\|_2^2\,\pi_N(z)\,dz
=\int_{\Theta}z\pi_N(z)\,dz,\\
\pi_N(z)&=\frac{\pi(z)\prod_{k=1}^Np(z,\xi^k)}
{\int_{\Theta}\pi(y)\prod_{k=1}^Np(y,\xi^k)\,dy}.
\end{split}
\end{equation}
The denominator is assumed finite and positive. The Bayesian counterparts of the preceding results take the following form \cite{gill1995applications}, \cite{vandervaart1998asymptotic}.
\begin{theorem}\label{Th:leCam2}
\begin{enumerate}
\item Assume that likelihood differentiation and integration by parts in the parameter are valid, with vanishing boundary terms. Suppose that the information matrices below are finite and $NI_p+I_\pi$ is positive definite. Then every estimator with finite integrated squared error satisfies the \textit{van Trees inequality}
\[
\EE\left[(\tilde x^N-X)(\tilde x^N-X)^T\right]
\succcurlyeq (NI_p+I_\pi)^{-1},
\]
where the expectation includes the prior draw and the observations, $I_X$ is the likelihood information matrix at parameter $X$, and
\[
I_p=\EE_X I_X,\qquad
I_\pi=\EE_X\left[\nabla\log\pi(X)\nabla\log\pi(X)^T\right].
\]
\item For fixed dimension, suppose the model is differentiable in quadratic mean at an interior $x^\star$, with nonsingular Fisher information, admits uniformly consistent tests of $x^\star$ against the complement of each neighborhood, and the prior has a continuous positive density at $x^\star$. The \textit{Bernstein--von Mises theorem} then gives an asymptotically Gaussian posterior, centered at an efficient estimator such as the regular MLE, with covariance $(NI_{x^\star})^{-1}$. If the first moments of the rescaled posterior errors are also uniformly integrable in probability, then
\[
\sqrt{N}(\hat x_{\mathrm B}^N-x^\star)\ \Longrightarrow\
\mathcal N(0,I_{x^\star}^{-1}),\qquad N\to\infty.
\]
\end{enumerate}
\end{theorem}
The prior-averaged bound and the fixed-parameter asymptotic statement have different assumptions and meanings. In particular, posterior normal approximation alone does not control its unbounded moments.

The \textit{maximum a posteriori estimator} (MAP) also has a direct optimization interpretation:
\[
\hat x_{\mathrm{MAP}}^N\in\operatorname{Arg}\min_{x\in\Theta}
\left[-\sum_{k=1}^N\log p(x,\xi^k)-\log\pi(x)\right].
\]
Thus the negative log-prior acts as a regularizer. In a normalized empirical objective its coefficient is $1/N$ for a fixed prior. In regular models, a consistent MAP with a smooth positive prior is often asymptotically equivalent to the MLE. At finite sample sizes, the posterior mean and mode need not coincide.

Let us consider two examples. The first is \textit{Regularized Least Squares}.
\begin{example}[Ridge Regression and LASSO]
Let  $x^\star \in \R^n$ be an unknown vector of parameters and  $\eta \sim \mathcal{N}\left(0,\sigma^2\right)$ be Gaussian noise. Assume that we can measure
\[
\xi^k = \langle a_k, x^\star\rangle  + \eta^k, \qquad k=1,...,N,
\]
where $\eta^k$ are i.i.d. (independent identically distributed as $\eta$) and the design matrix $A=\left[a_1,\ldots,a_N\right]^T\in\R^{N\times n}$ is known.\footnote{Note that $a_k$ can also be generated randomly. 
{In this case, to preserve the results it is sufficient to require  that $\left\{a_k\right\}_{k=1}^N$ and $\left\{\eta^k\right\}_{k=1}^N$ are independent.}
}
\textbf{The goal is to estimate $x^\star$ from $\xi:=\left\{\xi^k\right\}_{k=1}^N$.} 
Simple calculations lead to the following formulas\footnote{We assume that parameters $\sigma^2$, $\sigma^2_{\pi}$, $\lambda$ are known.}
\begin{align*}
    \hat{x}_{\mathrm{MLE}}^N &\in \operatorname{Arg}\min\limits_{x\in \R^n}\left[ \frac{1}{2\sigma^2}\|Ax - \xi\|_2^2\right],\\
    \hat{x}_{\mathrm B}^N = \hat{x}_{\mathrm{MAP}}^N &=\arg\min\limits_{x\in \R^n} \left[ \frac{1}{2\sigma^2}\|Ax - \xi\|_2^2 + \frac{1}{2\sigma_{\pi}^2}\|x - \bar{x}\|_2^2\right],
\end{align*}
where the prior coordinates $x_i$, $i=1,\ldots,n$, are independent with distributions $\mathcal{N}\left(\bar{x}_i,\sigma_{\pi}^2\right)$ (Ridge Regression) and  
$$\hat{x}_{\mathrm{MAP}}^N \in\operatorname{Arg}\min\limits_{x\in \R^n} \left[\frac{1}{2\sigma^2}\|Ax - \xi\|_2^2 + \lambda\|x\|_1\right],$$
where the prior probability density is  (LASSO): $$\pi(x)= \prod_{i=1}^n \frac{\lambda}{2}\exp\{-\lambda|x_i|\} = \left(\frac{\lambda}{2}\right)^n\exp\left\{-\lambda\|x\|_1\right\}.$$ 

For fixed prior parameters, the influence of the prior vanishes asymptotically when information accumulates in every direction, for example when
\[
\lambda_{\min}(A^TA)\longrightarrow\infty,
\qquad N\to\infty.
\]
Here the design matrix $A$ grows with the sample size. With full column rank, letting $\sigma_\pi^2\to\infty$ in the Gaussian-prior example recovers ordinary least squares. A rank-deficient design does not identify all coordinates, so neither conclusion holds without qualification. For the Laplace prior, the MAP is the LASSO solution; the posterior mean is generally different.

\end{example}

The second example connects regularization with classification.
\begin{example}[Soft-SVM]
Let $\xi^k=(y^k,a_k)$, with $y^k\in\{-1,1\}$, and introduce the hinge loss
\[
\ell(x,\xi^k)=\max\{0,1-y^k\langle x,a_k\rangle\}.
\]
The \textit{generalized Bayesian}, or Gibbs, distribution
\[
\pi_N(x)\propto\pi(x)\exp\left\{-\sum_{k=1}^N\ell(x,\xi^k)\right\}
\]
updates a prior directly through a loss \cite{bissiri2016general}. For a Gaussian prior with independent coordinates $\mathcal N(0,\sigma_\pi^2)$, its mode solves
\[
\operatorname{arg}\min_{x\in\R^n}\left[
\sum_{k=1}^N\max\{0,1-y^k\langle x,a_k\rangle\}
+\frac{1}{2\sigma_\pi^2}\|x\|_2^2\right],
\]
the usual soft-margin SVM objective. The aim is to learn a separating hyperplane while controlling its coefficient norm.

This is a loss-based update. Treating $\exp\{-\ell(x,(y,a))\}$ as an ordinary conditional likelihood would require division by
\[
Z(\langle x,a\rangle)=\sum_{y\in\{-1,1\}}
\exp\{-\max(0,1-y\langle x,a\rangle)\},
\]
which depends on $x$ and changes the objective. The Gibbs interpretation preserves precisely the hinge-loss formulation.
\end{example}

\subsection{Machine Learning motivation}\label{sec:ML}
In the statistical approach, the loss function is $f(x,\xi)=-\log p(x,\xi)$. This approach requires a parametric model $p(x,\xi)$. In many practical situations, such a model is not available. In \textit{regression problems}, however, an observation $\xi=(y,a)$ consists of a response $y$ and a feature vector $a$, and we can introduce the \textit{squared loss}
\[
f(x,\xi)=\left(y-\langle a,x\rangle\right)^2.
\]  Without specifying a parametric distribution for $\xi$, we can consider the expected loss minimization problem (stochastic optimization problem)
$$\min\limits_{x\in \R^n} \EE_{(y,a)}\left[\left(y-\langle a,x \rangle\right)^2\right].$$
In the offline approach, this problem has a form:
$$\min\limits_{x\in \R^n} \frac{1}{N}\|Y- Ax\|_2^2,$$
where $Y=\left(y^1,...,y^N\right)^T$, $A = \left[a_1,...,a_N\right]^T$.
Similarly, in \textit{classification problems} with $y\in\{-1,1\}$, the \textit{hinge loss} is
\[
f(x,\xi)=\max\left\{0,1-y\langle x,a\rangle\right\},
\]
and the corresponding stochastic optimization problem has the form
$$\min\limits_{x\in \R^n} \EE_{(y,a)}\left[\max\left\{0,1 - y\langle x,a\rangle\right\}\right].$$
In many applications, prior information suggests a bound on the size of $x^\star$. Typically, this information is formalize{d} as a constraint of the type $x\in Q$, where $Q$ is often chosen as a ball $B_p^n(R_p)$ in $p$-norm ($p\ge 1$) centered at $0$ with radius $R_p$  or another convex compact set with simple structure, e.g. unit simplex $S_n(1)$. So the final stochastic optimization problem in general has the form \eqref{eq:SOP}\footnote{Here and everywhere below we will denote the solution of this problem as $x^\star$. If the solution is not unique $x^\star$ means one of the solutions, e.g. such that is the closest to the starting point (initial guess).}
For a norm-ball constraint, convex duality can connect the constrained problem to norm regularization: under a constraint qualification, a Lagrange multiplier gives an equivalent penalized problem at a constrained solution. This connection motivates both formulations, although a regularization parameter must still be selected.

The regression and classification examples have useful common structure, subject to appropriate assumptions on the data. For the estimates below we use:
\begin{enumerate}
\item Convexity of $f(\cdot,\xi)$ and a measurable subgradient $g(x,\xi)\in\partial_x f(x,\xi)$ with $\|g(x,\xi)\|_2\le M_2$ on $Q$, almost surely. In particular,
\[
|f(z,\xi)-f(x,\xi)|\le M_2\|z-x\|_2.
\]
The notation $\nabla_x f$ below also denotes such a subgradient when the loss is nonsmooth.
\item A generalized linear structure, possibly with a deterministic convex regularizer:
\[
f(x,\xi)=g_0\left(y(\xi),\langle x,a(\xi)\rangle\right)+r(x),
\qquad \|a(\xi)\|_2\le B,
\]
where $g_0(y,\cdot)$ is $L_0$-Lipschitz on the relevant prediction interval, uniformly in $y$, and $L_0B\le M_2$.
\end{enumerate}
We assume integrability of the loss at one fixed point of $Q$, so all expected risks are finite. Hinge loss with bounded features satisfies these assumptions on bounded $Q$, with a suitable bounded subgradient of the regularizer. Squared loss also does so for bounded features and responses. Gaussian responses, however, do not give a deterministic uniform bound on the stochastic gradient; they require a separate moment or tail analysis.

The first condition supports the online approach. The additional linear structure gives dimension-independent statistical bounds for the offline approach. A strongly convex regularizer is particularly useful: an individual unregularized linear-prediction loss has flat directions orthogonal to its feature vector and is generally not strongly convex in the full parameter space.

Let us start with the \textit{offline} approach. We introduce the \textit{empirical loss}
$$\bar{f}(x):=\bar{f}\left(x,\left\{\xi^k\right\}_{k=1}^N\right) = \frac{1}{N}\sum_{k=1}^N f(x,\xi^k)$$
and its minimizer 
$$\hat{x}^N \in \text{Arg}\min\limits_{x\in Q} \bar{f}\left(x,\left\{\xi^k\right\}_{k=1}^N\right).$$

\begin{theorem}[Learnability for generalized linear models]\label{Th:GLC}
Let $Q\subseteq B_2^n(R_2)$ be a closed convex set, let the observations be i.i.d., and assume conditions 1 and 2 above. Fix a deterministic $x_0\in Q$. With probability at least $1-\beta$, $0<\beta<1$,
\[
\sup_{x\in Q}\left|
[\bar f(x)-\bar f(x_0)]-[f(x)-f(x_0)]
\right|
=\cO\left(M_2R_2\sqrt{\frac{\log(2/\beta)}N}\right).
\]
Consequently, on the same event, simultaneously for all $x\in Q$,
\begin{equation}\label{eq:GLC}
f(x)-f(x^\star)\le\bar f(x)-\bar f(\hat x^N)
+\cO\left(M_2R_2\sqrt{\frac{\log(2/\beta)}N}\right).
\end{equation}
If additionally the population objective $f$ is $\mu_2$-strongly convex, then, with probability at least $1-\beta$, simultaneously for all $x\in Q$,
\begin{equation}\label{eq:GLSC}
f(x)-f(x^\star)\le2[\bar f(x)-\bar f(\hat x^N)]
+\cO\left(\frac{M_2^2\log(2/\beta)}{\mu_2N}\right).
\end{equation}
For general losses without condition 2, suppose instead that every $f(\cdot,\xi)$ is $\mu_2$-strongly convex and condition 1 holds. Then, with probability at least $1-\beta$, simultaneously for all $x\in Q$,
\begin{equation}\label{eq:GLSC2}
f(x)-f(x^\star)\le
M_2\sqrt{\frac{2[\bar f(x)-\bar f(\hat x^N)]}{\mu_2}}
+\cO\left(\frac{M_2^2\log(2N)\log(2/\beta)}{\mu_2N}\right).
\end{equation}
\end{theorem}
The first bound follows from symmetrization and contraction for linear predictors \cite{shalev2014understanding}. Centering removes arbitrary random additive constants in the loss; the deterministic regularizer cancels from the empirical fluctuation. The fast bound \eqref{eq:GLSC} is a specialization of \cite[Theorem~1]{sridharan2008fast}. For \eqref{eq:GLSC2}, combine the excess-risk bound for the exact empirical minimizer \cite[Proposition~2.1]{klochkov2021stability} with
\[
\|x-\hat x^N\|_2\le
\sqrt{\frac{2[\bar f(x)-\bar f(\hat x^N)]}{\mu_2}}
\]
and Lipschitz continuity of $f$. This also explains the stronger empirical accuracy requirement in the last inequality. The bounds concern the full objective, including $r$ when present; removing the regularizer changes the statistical target.

This theorem reduces stochastic optimization problem to the empirical loss (risk) minimization problem 
\begin{equation}\label{eq:ERM}
  \min\limits_{x\in Q}  \frac{1}{N}\sum_{k=1}^N f(x,\xi^k)  
\end{equation}
with proper choice of $N$, see the next section.

Now we move to the \textit{online} approach. At iteration $k$, the method uses the current point and a new observation:
\begin{equation}\label{eq:projSGD}
x^{k+1}=\pi_Q\left(x^k-\gamma_k\nabla_x f(x^k,\xi^k)\right),
\qquad k=1,\ldots,N,
\end{equation}
where $\pi_Q$ is the Euclidean projection onto the closed convex set $Q$, and $x^1\in Q$. Let $D_2$ be a known upper bound on the diameter of $Q$; for $Q\subseteq B_2^n(R_2)$ we may take $D_2=2R_2$.

For any comparator $u\in Q$, nonexpansiveness of projection and $\pi_Q(u)=u$ give
\begin{align*}
\|x^{k+1}-u\|_2^2
&=\|\pi_Q(x^k-\gamma_k\nabla_x f(x^k,\xi^k))-\pi_Q(u)\|_2^2\\
&\le\|x^k-u\|_2^2
-2\gamma_k\langle\nabla_x f(x^k,\xi^k),x^k-u\rangle
+\gamma_k^2M_2^2.
\end{align*}
Convexity therefore implies
\[
f(x^k,\xi^k)-f(u,\xi^k)
\le\frac{\|x^k-u\|_2^2-\|x^{k+1}-u\|_2^2}{2\gamma_k}
+\frac{\gamma_kM_2^2}{2}.
\]
For a constant stepsize $\gamma$, summation yields the useful general bound
\[
\frac{1}{N}\sum_{k=1}^N[f(x^k,\xi^k)-f(u,\xi^k)]
\le\frac{\|x^1-u\|_2^2}{2\gamma N}+\frac{\gamma M_2^2}{2}.
\]
Choosing
\[
\gamma=\frac{D_2}{M_2\sqrt{N}}
\]
gives, simultaneously for all $u\in Q$,
\begin{equation}\label{eq:onlineConvex}
\frac{1}{N}\sum_{k=1}^N f(x^k,\xi^k)
\le\frac{1}{N}\sum_{k=1}^N f(u,\xi^k)+\frac{M_2D_2}{\sqrt N}.
\end{equation}
The radius used in the bound is the one used to set the stepsize; it cannot subsequently be replaced by an unknown, possibly zero, distance to a minimizer.

If every sample loss is $\mu_2$-strongly convex, the same projection inequality gives the stronger one-step bound
\begin{align*}
f(x^k,\xi^k)-f(u,\xi^k)
&\le\langle\nabla_x f(x^k,\xi^k),x^k-u\rangle
-\frac{\mu_2}{2}\|x^k-u\|_2^2\\
&\le\left(\frac{1}{2\gamma_k}-\frac{\mu_2}{2}\right)\|x^k-u\|_2^2
-\frac{\|x^{k+1}-u\|_2^2}{2\gamma_k}
+\frac{\gamma_kM_2^2}{2}.
\end{align*}
For $\gamma_k=1/(\mu_2k)$, summing over $k=1,\ldots,N$ cancels consecutive squared distances:
\begin{align*}
\sum_{k=1}^N[f(x^k,\xi^k)-f(u,\xi^k)]
&\le\frac{\mu_2}{2}\sum_{k=1}^N
\bigl[(k-1)\|x^k-u\|_2^2-k\|x^{k+1}-u\|_2^2\bigr]
+\frac{M_2^2}{2\mu_2}\sum_{k=1}^N\frac{1}{k}\\
&=-\frac{\mu_2N}{2}\|x^{N+1}-u\|_2^2
+\frac{M_2^2}{2\mu_2}\sum_{k=1}^N\frac{1}{k}
\le\frac{M_2^2(1+\log N)}{2\mu_2}.
\end{align*}
Dividing by $N$ yields
\begin{equation}\label{eq:onlineSConvex}
\frac{1}{N}\sum_{k=1}^N f(x^k,\xi^k)
\le\frac{1}{N}\sum_{k=1}^N f(u,\xi^k)
+\frac{M_2^2(1+\log N)}{2\mu_2N}.
\end{equation}
No probabilistic assumption on the observations was used here. Taking the best comparator in hindsight characterizes SGD as an online learning procedure in the standard regret sense \cite{cesa2006prediction}.

For i.i.d. observations, these inequalities give statistical guarantees by the standard conversion from online to batch learning; see also \cite{kakade2008generalization}.
\begin{theorem}\label{Th:Online_Prob}
Assume condition 1, i.i.d. observations, $\operatorname{diam}_2(Q)\le D_2$, and $0<\beta<1$. Let the iterates follow \eqref{eq:projSGD}, with each $x^k$ computed before observing $\xi^k$. With constant stepsize $D_2/(M_2\sqrt N)$, with probability at least $1-\beta$,
\begin{equation}\label{eq:COR}
f(\bar x^N)-f(x^\star)
\le\frac{M_2D_2}{\sqrt N}\left(1+\sqrt{2\log(1/\beta)}\right),
\end{equation}
where
\begin{equation}\label{eq:avg}
\bar x^N=\frac{1}{N}\sum_{k=1}^Nx^k.
\end{equation}
If each sample loss is also $\mu_2$-strongly convex, then the stepsizes $\gamma_k=1/(\mu_2k)$ give, with probability at least $1-\beta$,
\begin{equation}\label{eq:SCOR}
f(\bar x^N)-f(x^\star)
\le\frac{M_2^2}{\mu_2N}
\left[1+\log N+4\log(1/\beta)\right].
\end{equation}
\end{theorem}
To see the conversion, introduce the sample and population excess losses
\[
Z_k=f(x^k,\xi^k)-f(x^\star,\xi^k),
\qquad
G_k=f(x^k)-f(x^\star),
\qquad k=1,\ldots,N.
\]
Let $\mathcal F_{k-1}$ denote the information available before drawing $\xi^k$, including the initial point. Conditional on this past,
\[
G_k=\EE[Z_k\mid\mathcal F_{k-1}],
\qquad |Z_k|\le M_2D_2.
\]
Define the full sums
\[
S_N=\sum_{k=1}^NG_k,
\qquad
V_N=\sum_{k=1}^N(G_k-Z_k).
\]
Hoeffding's conditional inequality yields, with probability at least $1-\beta$,
\[
V_N\le M_2D_2\sqrt{2N\log(1/\beta)}.
\]
Since
\[
S_N=\sum_{k=1}^NZ_k+V_N,
\qquad
f(\bar x^N)-f(x^\star)\le\frac{S_N}{N},
\]
we obtain \eqref{eq:COR} from \eqref{eq:onlineConvex}.

In the strongly convex case one can exploit the smaller, data-dependent range:
\[
|Z_k|\le M_2\|x^k-x^\star\|_2,
\qquad
\|x^k-x^\star\|_2^2\le\frac{2G_k}{\mu_2}.
\]
Conditional Hoeffding and successive conditioning give, for every $t>0$,
\[
\EE\exp\left\{tV_N-\frac{t^2M_2^2}{\mu_2}S_N\right\}\le1.
\]
Set
\[
t=\frac{\mu_2}{2M_2^2}.
\]
Markov's inequality implies, with probability at least $1-\beta$,
\[
V_N\le\frac{S_N}{2}
+\frac{2M_2^2}{\mu_2}\log(1/\beta).
\]
Combining the identity for $S_N$ with \eqref{eq:onlineSConvex}, we find
\[
S_N\le\frac{M_2^2(1+\log N)}{2\mu_2}
+\frac{S_N}{2}+\frac{2M_2^2}{\mu_2}\log(1/\beta),
\]
which proves \eqref{eq:SCOR} after rearrangement and division by $N$.

With $D_2=2R_2$, the online and offline bounds have the same basic sample dependence, $N^{-1/2}$ in the convex case and $N^{-1}$ up to logarithms in the strongly convex case. Their computational requirements differ. The online procedure processes samples directly, whereas the offline procedure must solve an empirical problem to the accuracy prescribed by \eqref{eq:GLC}--\eqref{eq:GLSC2}. A more detailed comparison is given in the next section.

To conclude this section, the statistical approach is a particular case of the general machine learning approach, with the loss determined by a likelihood model. We shall therefore mainly use the language of expected risk minimization and the stochastic optimization formulation \eqref{eq:SOP}.

\section{Sample Average Approximation versus Stochastic Approximation}\label{Sec:SAAvsSA}
In this section, we consider stochastic optimization problem \eqref{eq:SOP}.
We are mainly interested in the sample complexity of offline (also called \textit{Sample Average Approximation}) and online (also called \textit{Stochastic Approximation}) procedures, which generate $\tilde{x}^N$ from the empirical problem \eqref{eq:ERM} or from a procedure of type \eqref{eq:projSGD}. More precisely, we seek a sample size $N=N(\varepsilon,\beta)$ such that
\begin{equation}\label{eq:P_tilde}
\mathds{P}\left(f\left(\tilde{x}^N\right)-f(x^*)\leq\varepsilon\right)\geq 1-\beta.
\end{equation}
In addition to the standing conventions, assume $Q\subseteq B_p^n(R_p)$, $p\geq1$, unless stated otherwise. The general covering result below also allows nonconvex compact $Q$. For all $x,y\in Q$, almost surely,
\begin{equation}\label{eq:Lip_p}
|f(y,\xi)-f(x,\xi)|\leq M_p\|y-x\|_p.
\end{equation}
This condition ensures continuity of both risks and existence of minima on compact $Q$. We take $\varepsilon>0$, $0<\beta<1$, round sample sizes upwards, and require at least one sample. Relevant geometric factors will be specified explicitly.

Let $\bar{x}_{\delta,\tilde\beta}^N$ be a $(\delta,\tilde\beta)$-solution of the empirical problem
\[
\min_{x\in Q}\left\{\bar f(x):=\frac1N\sum_{k=1}^N f(x,\xi^k)\right\},
\]
that is, with probability at least $1-\tilde\beta$,
\[
\bar f(\bar{x}_{\delta,\tilde\beta}^N)-\bar f(\hat x^N)\leq\delta,
\qquad \hat x^N\in\argmin_{x\in Q}\bar f(x).
\]
The probability includes any randomness of the empirical solver. Sample complexity counts independent observations; the work needed to solve the resulting empirical problem is a separate question.

\subsection{Non-convex case and convex case}\label{Sec:NonConv_Conv}
One of the basic results about the offline approach is the following covering estimate; see \cite{shapiro2005complexity,shapiro2021lectures}.

\begin{theorem}\label{Th:ShapiroNonConvex}
Assume \eqref{eq:Lip_p}. Convexity is not required. Then \eqref{eq:P_tilde} holds for $\tilde x^N=\bar{x}_{\varepsilon/2,\beta/2}^N$ with a sample size of order
\begin{equation}\label{eq:N_ERM_nonConv}
N(\varepsilon,\beta)=\cO\left(\frac{M_p^2R_p^2}{\varepsilon^2}
\left[n\log\left(1+\frac{M_pR_p}{\varepsilon}\right)+\log\frac{2}{\beta}\right]\right).
\end{equation}
This is a general sufficient bound. Convexity allows a sharper analysis of empirical minimizers than this uniform covering argument.
\end{theorem}
\begin{proof}
Fix a population minimizer $x^*$ and center the losses at this point:
\[
h_x(\xi)=f(x,\xi)-f(x^*,\xi),\qquad
Z_N(x)=\frac1N\sum_{k=1}^N h_x(\xi^k)-\EE h_x(\xi).
\]
Then $|h_x(\xi)|\leq2M_pR_p$. A maximal $\rho$-separated subset $\mathcal N\subset Q$, with $\rho=\varepsilon/(8M_p)$, is a $\rho$-net and satisfies
\begin{equation}\label{eq:net}
|\mathcal N|\leq\left(1+\frac{2R_p}{\rho}\right)^n
=\left(1+\frac{16M_pR_p}{\varepsilon}\right)^n.
\end{equation}
For every $z\in\mathcal N$, Hoeffding's inequality gives
\[
\mathds P\{|Z_N(z)|>\varepsilon/4\}
\leq2\exp\left(-\frac{N\varepsilon^2}{128M_p^2R_p^2}\right).
\]
A union bound therefore controls all net points with probability at least $1-\beta/2$ when
\[
N\geq\frac{128M_p^2R_p^2}{\varepsilon^2}
\left[n\log\left(1+\frac{16M_pR_p}{\varepsilon}\right)+\log\frac4\beta\right].
\]
For $\|x-z\|_p\leq\rho$, Lipschitz continuity gives
$|Z_N(x)-Z_N(z)|\leq2M_p\rho=\varepsilon/4$.
Thus $\sup_{x\in Q}|Z_N(x)|\leq\varepsilon/2$ on the same event. Every empirical $\varepsilon/2$-solution then satisfies
\[
f(\tilde x^N)-f(x^*)
\leq\bar f(\tilde x^N)-\bar f(x^*)+\varepsilon/2
\leq\varepsilon.
\]
Adding the solver's failure probability $\beta/2$ proves the claim. Centering is useful here: a random additive constant in the loss has no effect on either minimizer and need not be bounded.
\end{proof}

\begin{remark}\label{Feldman}
For convex Lipschitz losses on a Euclidean ball, the worst-case sample complexity of the ERM principle is now understood. In the normalization $M_2=R_2=1$ and at fixed confidence, \cite{pmlr-v238-carmon24a} proves the upper bound
\[
N=\tilde{\cO}\left(\frac n\varepsilon+\frac1{\varepsilon^2}\right),
\]
matching the lower bound of \cite{feldman2016generalization} up to logarithms. This is a guarantee for every empirical minimizer, and also for empirical solutions of sufficiently small constant-multiple accuracy $\delta=\cO(\varepsilon)$. It is stronger than the uniform convergence bound $\tilde{\cO}(n/\varepsilon^2)$. Thus the dimension dependence of unregularized ERM and the dimension dependence of uniform convergence should be distinguished. The result applies to a larger class than smooth convex losses, so \eqref{eq:N_ERM_nonConv} is not an optimal general bound for the latter class either.
\end{remark}

For the online approach, the geometry of the feasible set enters through the choice of a mirror map \cite{nemirovski1983problem,nemirovski2009robust}. Recall that $q$ is the conjugate exponent, $1/p+1/q=1$. We assume access to measurable stochastic subgradients $g(x,\xi)$ with
\[
\EE[g(x,\xi)]\in\partial f(x),\qquad \|g(x,\xi)\|_q\leq M_p
\]
almost surely. Each iteration uses a fresh observation, independent of the past. This assumption specifies the oracle actually used; Lipschitz continuity on a constrained set alone does not bound every possible subgradient at its boundary.

\begin{theorem}\label{Th:Agarwal}
Suppose the losses are convex and \eqref{eq:Lip_p} and the above oracle assumptions hold. For $1\leq p\leq2$, stochastic Mirror Descent with averaging satisfies \eqref{eq:P_tilde} with
\[
N(\varepsilon,\beta)=\cO\left(\frac{K_pM_p^2R_p^2}{\varepsilon^2}\log\frac2\beta\right),
\]
where
\[
K_p=\begin{cases}
\log(en),&p=1,\\
\min\{(p-1)^{-1},\log(en)\},&1<p\leq2.
\end{cases}
\]
For $p>2$, Euclidean projected stochastic subgradient descent gives
\[
N(\varepsilon,\beta)=\cO\left(n^{1-2/p}\frac{M_p^2R_p^2}{\varepsilon^2}\log\frac2\beta\right).
\]
If a smaller Euclidean oracle bound $\|g(x,\xi)\|_2\leq M_2$ is available, $M_p$ in the last formula can be replaced by $M_2$.
\end{theorem}
To specify the geometry, let $\psi$ be $1$-strongly convex in the chosen norm and use the Bregman divergence
\[
V_\psi(y,x)=\psi(y)-\psi(x)-\langle\nabla\psi(x),y-x\rangle.
\]
For a fresh observation $\xi^k$, write $g_k=g(x^k,\xi^k)$ and let $\gamma_k>0$ be the stepsize. The Mirror Descent update is
\[
x^{k+1}\in\argmin_{x\in Q}\left\{\gamma_k\langle g_k,x\rangle+V_\psi(x,x^k)\right\}.
\] For $1<p\leq2$, the map $\|x\|_p^2/[2(p-1)]$ has the required strong convexity. Near $p=1$, a suitably scaled squared $r$-norm with $r=1+1/\log(en)$ gives a radius of order $R_p^2\log(en)$ instead. Choosing the better of these maps yields $\sup_{x\in Q}V_\psi(x,x^1)=\cO(K_pR_p^2)$ when $x^1$ minimizes the map on $Q$. For $p>2$, use $\psi(x)=\|x\|_2^2/2$ and the inclusion $B_p^n(R_p)\subseteq B_2^n(n^{1/2-1/p}R_p)$; also $\|g(x,\xi)\|_2\leq\|g(x,\xi)\|_q$.

In particular, the Euclidean sample bound is dimension independent. Its $\varepsilon^{-2}$ dependence is optimal for the usual bounded stochastic first-order oracle \cite{agarwal2012information}. Such oracle lower bounds specify what information each call reveals; they are not lower bounds on every procedure that can inspect an entire sampled function. The sharper convex ERM bound in Remark~\ref{Feldman} reduces the apparent gap between the two approaches. Proper regularization removes the remaining dimension dependence from the statistical guarantee, as we explain next.

\subsection{Strongly convex case and regularization}\label{Sec:SCR}
If $f(x,\xi)$ is \textit{$\mu_p$-strongly convex} in the $p$-norm, then, for all $x,y\in Q$ and an appropriate subgradient $g(x,\xi)$,
\begin{equation}\label{eq:Scv}
f(y,\xi)\geq f(x,\xi)+\langle g(x,\xi),y-x\rangle
+\frac{\mu_p}{2}\|y-x\|_p^2.
\end{equation}
Both the empirical and the population problems are then strongly convex. Their statistical behavior is better than the general covering bound suggests.

\begin{theorem}\label{Th:StrCnv}
Assume \eqref{eq:Lip_p} and \eqref{eq:Scv}. There is a universal constant $C$ such that \eqref{eq:P_tilde} holds for $\tilde x^N=\bar x_{\delta,\beta/2}^N$ whenever
\begin{equation}\label{eq:N_ERM_scv}
\delta\leq\frac{\mu_p\varepsilon^2}{8M_p^2},\qquad
\frac{CM_p^2}{\mu_p N}\log(2N)\log\frac4\beta\leq\frac\varepsilon2.
\end{equation}
Thus the offline sample complexity is dimension independent for any fixed norm:
\[
N_{\mathrm{SAA}}(\varepsilon,\beta)
=\widetilde{\cO}\left(\frac{M_p^2}{\mu_p\varepsilon}\right).
\]
For $1\leq p\leq2$, restarted stochastic Mirror Descent under the oracle assumptions of Theorem~\ref{Th:Agarwal} has online sample complexity
\[
N_{\mathrm{SA}}(\varepsilon,\beta)
=\widetilde{\cO}\left(\frac{K_pM_p^2}{\mu_p\varepsilon}\right).
\]
The subscripts distinguish the two procedures, and the suppressed logarithmic factors include the confidence dependence on $\beta$. In the Euclidean case the dependence on $M_2,\mu_2,\varepsilon$ agrees with the stochastic first-order lower bound up to logarithmic factors.
\end{theorem}
\begin{proof}
The offline argument is based on uniform stability \cite{bousquet2002stability,shalev2009stochastic}. Let $S=(\xi^1,\ldots,\xi^N)$ and let $S^{(i)}$ replace observation $\xi^i$ by an independent copy, for $i\in\{1,\ldots,N\}$. Write
\[
\bar f_S(x)=\frac1N\sum_{k=1}^N f(x,\xi^k),\qquad
x_S\in\argmin_{x\in Q}\bar f_S(x),
\]
and define $\bar f_{S^{(i)}}(x)$ and $x_{S^{(i)}}$ in the same way using the replaced sample. Adding the strong convexity inequalities at these minima, and writing $d=\|x_S-x_{S^{(i)}}\|_p$, gives
\begin{equation}\label{eq:un_stab_pre}
\mu_p d^2\leq\bar f_S(x_{S^{(i)}})-\bar f_S(x_S)
+\bar f_{S^{(i)}}(x_S)-\bar f_{S^{(i)}}(x_{S^{(i)}})
\leq\frac{2M_p}{N}d.
\end{equation}
Consequently $d\leq2M_p/(\mu_pN)$, and replacing one observation changes the loss at the output by at most $2M_p^2/(\mu_pN)$. Replacing that observation by an independent copy and using exchangeability yields
\begin{equation}\label{eq:expect_bound}
\EE[f(\hat x^N)-f(x^*)]\leq
\EE[f(\hat x^N)-\bar f(\hat x^N)]
\leq\frac{2M_p^2}{\mu_pN}.
\end{equation}
For high probability, the excess-risk stability result of \cite[Theorem 1.1 and Proposition 2.1]{klochkov2021stability} gives
\begin{equation}\label{eq:ERM_stability_hp}
f(\hat x^N)-f(x^*)\leq
\frac{CM_p^2}{\mu_p N}\log(2N)\log\frac4\beta
\end{equation}
with probability at least $1-\beta/2$. Its assumptions can be checked directly in the present norm. Strong convexity and Lipschitz continuity imply $\operatorname{diam}_p(Q):=\sup_{x,y\in Q}\|x-y\|_p\leq4M_p/\mu_p$. Centering the loss at a fixed point and adding a constant therefore makes it bounded and nonnegative, without changing any minimizer. Moreover,
\[
\EE[(f(x,\xi)-f(x^*,\xi))^2]
\leq M_p^2\|x-x^*\|_p^2
\leq\frac{2M_p^2}{\mu_p}[f(x)-f(x^*)],
\]
which is the Bernstein condition used by that result. The stability, boundedness and Bernstein constants are all of order $M_p^2/\mu_p$ after accounting for the factor $1/N$ in stability.

Finally, strong convexity of $\bar f$ and Lipschitz continuity of $f$ give, for every empirical $\delta$-solution $x$, the deterministic transfer
\begin{equation}\label{eq:emp_bound}
f(x)-f(x^*)\leq
M_p\sqrt{\frac{2\delta}{\mu_p}}+f(\hat x^N)-f(x^*).
\end{equation}
Combining \eqref{eq:ERM_stability_hp} and \eqref{eq:emp_bound}, and adding the solver's failure probability, proves the offline claim. Taking expectations instead gives
\[
\EE[f(x)-f(x^*)]\leq M_p\sqrt{\frac{2\delta}{\mu_p}}+\frac{2M_p^2}{\mu_pN}
\]
when the empirical accuracy holds surely. The online claim follows by restarting the convex method, using $f(x)-f(x^*)\geq(\mu_p/2)\|x-x^*\|_p^2$; see Theorem~\ref{Th:r_online} and \cite{juditsky2014deterministic}.
\end{proof}

The difference between the empirical accuracies in Theorems~\ref{Th:ShapiroNonConvex} and~\ref{Th:StrCnv} is substantial. The first allows $\delta$ of order $\varepsilon$, whereas the dimension-independent strongly convex guarantee uses $\delta$ of order $\mu_p\varepsilon^2/M_p^2$. This quadratic dependence cannot generally be weakened for an arbitrary approximate ERM solver. For example, on $B_2^n(a/\mu)$ consider
\[
f(x,\alpha)=a\|\alpha\odot x\|_2+\frac\mu2\|x\|_2^2,
\]
where $a,\mu>0$, the coordinates of $\alpha$ are independent Bernoulli$(1/2)$ variables, and $\odot$ denotes coordinatewise multiplication. Let $e_j$ denote the $j$th coordinate vector. The losses are $\mu$-strongly convex and $2a$-Lipschitz. If coordinate $j$ has not appeared in the sample, then $x=\sqrt{2\delta/\mu}\,e_j$, when feasible, has empirical error $\delta$ and population error $a\sqrt{\delta/(2\mu)}+\delta$. The probability that some coordinate is unseen is $1-(1-2^{-N})^n$, which can be close to one in sufficiently large dimension. This is the mechanism behind the approximate-ERM obstruction discussed in \cite{shalev2009stochastic}.

Based on Theorem~\ref{Th:StrCnv}, regularization improves the general convex bound. For clarity, first assume $p=2$.
\begin{lemma}[Tikhonov regularization]\label{Lm:reg}
Assume $R_2>0$ (the singleton case is trivial), and consider
\begin{equation}\label{eq:regularization}
\min_{x\in Q}\left\{f_\mu(x):=\EE_\xi f(x,\xi)+\frac\mu2\|x\|_2^2\right\},
\qquad \mu=\frac\varepsilon{R_2^2}.
\end{equation}
If
\[
f_\mu(\tilde x)-\min_{x\in Q}f_\mu(x)\leq\frac{\varepsilon}{2},
\]
then
\[
f(\tilde x)-\min_{x\in Q}f(x)\leq\varepsilon.
\]
\end{lemma}
\begin{proof}
Indeed,
\[
f(\tilde x)-f(x^*)\leq f_\mu(\tilde x)-\min_{x\in Q}f_\mu(x)
+\frac\mu2\|x^*\|_2^2
\leq\frac\varepsilon2+\frac\varepsilon2=\varepsilon.
\]
\end{proof}

\begin{theorem}[The role of regularization]\label{Th:conv_reg}
Assume convexity and \eqref{eq:Lip_p} with $p=2$, and $0<\varepsilon\leq M_2R_2$. Set $\mu=\varepsilon/R_2^2$ and $M_\mu=M_2+\mu R_2\leq2M_2$. Let $\tilde x^N$ solve
\[
\min_{x\in Q}\left\{\frac1N\sum_{k=1}^N f(x,\xi^k)+\frac\mu2\|x\|_2^2\right\}
\]
to empirical accuracy
\[
\delta\leq\frac{\mu\varepsilon^2}{32M_\mu^2}
=\frac{\varepsilon^3}{32M_\mu^2R_2^2}
\]
with failure probability at most $\beta/2$. Then \eqref{eq:P_tilde} holds provided
\[
\frac{CM_\mu^2}{\mu N}\log(2N)\log\frac4\beta\leq\frac\varepsilon4.
\]
In particular, the following sample size and empirical accuracy suffice:
\[
N(\varepsilon,\beta)=\widetilde{\cO}\left(\frac{M_2^2R_2^2}{\varepsilon^2}\right),\qquad
\delta=\frac{\varepsilon^3}{128M_2^2R_2^2}.
\]
\end{theorem}
The error budget is transparent: regularization contributes at most $\varepsilon/2$, the empirical solver at most $\varepsilon/4$ by \eqref{eq:emp_bound}, and the statistical error at most $\varepsilon/4$. The Lipschitz constant in the strongly convex estimate is that of the full regularized loss, namely $M_\mu$.

For $1\leq p\leq2$, the same construction uses a strongly convex mirror map in place of $\|x\|_2^2/2$. Its range determines the regularization bias; the maps used in Theorem~\ref{Th:Agarwal} give the corresponding sample bounds up to logarithmic and stated geometric factors. Thus online optimization and suitably regularized offline optimization have comparable sample complexity. This statistical equivalence does not imply equal computational cost.

\subsection{\texorpdfstring{$s$}{s}-growth condition}\label{Sec:rgrwt}
Strong convexity can be relaxed by relating the objective error to distance from the solution set. Write
\[
S=\argmin_{x\in Q}f(x),\qquad f^*=\min_{x\in Q}f(x),\qquad
\operatorname{dist}_p(x,S)=\inf_{y\in S}\|x-y\|_p.
\] We say that $f$ satisfies an \textit{$s$-growth condition}, $s\geq1$, on $Q_{2\varepsilon}=\{x\in Q:f(x)\leq f^*+2\varepsilon\}$ if
\begin{equation}\label{eq:r}
f(x)-f^*\geq\mu_{p,s}\operatorname{dist}_p(x,S)^s,
\qquad x\in Q_{2\varepsilon},\quad \mu_{p,s}>0.
\end{equation}
This condition allows multiple minimizers. When $S=\{x^*\}$, the distance is simply $\|x-x^*\|_p$ and the entire sublevel set lies in the ball of radius $(2\varepsilon/\mu_{p,s})^{1/s}$ about $x^*$.

For the offline result, replace uniform Lipschitz continuity by a condition on centered increments. Let
$Y(x,y,\xi)=f(y,\xi)-f(x,\xi)$. Assume that, for some $\lambda\geq0$, every $t\in\R$, and every required pair $x,y$,
\begin{equation}\label{eq:var}
\EE_\xi\exp\{t[Y(x,y,\xi)-\EE_\xi Y(x,y,\xi)]\}
\leq\exp\left(\frac{t^2\lambda^2\|y-x\|_p^2}{2}\right).
\end{equation}
Thus the increments have sub-Gaussian fluctuations with scale proportional to distance. Condition~\eqref{eq:Lip_p} implies this with $\lambda=M_p$ by Hoeffding's lemma, but Gaussian linear perturbations also satisfy it without a uniform Lipschitz bound.

\begin{theorem}\label{Th:ShapiroNonConvex2}
Let $Q$ be closed and convex, possibly unbounded. Assume the population risk and the sample losses are finite and continuous on $Q$, the sample losses are convex almost surely, and $f$ has a unique minimizer $x^*$. Suppose \eqref{eq:r} holds on $Q_{2\varepsilon}$ and \eqref{eq:var} holds for all $x,y\in Q_{2\varepsilon}$. Then \eqref{eq:P_tilde} holds for $\tilde x^N=\bar x_{\varepsilon/2,\beta/2}^N$ with
\begin{equation}\label{eq:N_ERM_r}
N(\varepsilon,\beta)=\cO\left(
\frac{\lambda^2}{\mu_{p,s}^{2/s}\varepsilon^{2(s-1)/s}}
\left[n+\log\frac2\beta\right]\right).
\end{equation}
For a sharp minimum, $s=1$, this sufficient sample size has no dependence on $\varepsilon$ if the same growth and increment constants apply throughout the accuracy range considered.
\end{theorem}
\begin{proof}
Put $r_\varepsilon=(2\varepsilon/\mu_{p,s})^{1/s}$. Growth and uniqueness give $Q_{2\varepsilon}\subseteq B_p(x^*,r_\varepsilon)$, a compact sublevel set. The centered process
\[
Z_N(x)=[\bar f(x)-\bar f(x^*)]-[f(x)-f(x^*)]
\]
has sub-Gaussian increments of scale $\lambda\|x-y\|_p/\sqrt N$. The usual multiscale version of the covering argument yields, with probability at least $1-\beta/2$,
\begin{equation}\label{eq:localized_process}
\sup_{x\in Q_{2\varepsilon}}|Z_N(x)|
\leq C\lambda r_\varepsilon\sqrt{\frac{n+\log(4/\beta)}{N}}.
\end{equation}
For completeness, take nets of radii $r_\varepsilon2^{-j}$ and cardinalities at most $(1+2^{j+1})^n$, for $j=0,1,2,\ldots$. Join each net point to a nearest point at the preceding scale, starting with $x^*$. A union bound over these links, with failure probabilities proportional to $\beta/(j+1)^2$, bounds their increments by
\[
\frac{C\lambda r_\varepsilon}{\sqrt N}\,2^{-j}
\sqrt{n(j+1)+\log\frac{2(j+1)^2}{\beta}}.
\]
Summing the convergent series proves \eqref{eq:localized_process}; continuity extends the bound from nets to all points. This argument uses the increment condition directly and requires no extra random Lipschitz constant.

Choose $N$ so that the right-hand side of \eqref{eq:localized_process} is at most $\varepsilon/2$. If an empirical $\delta$-solution $x$, $\delta\leq\varepsilon/2$, lies in $Q_{2\varepsilon}$, then
\[
f(x)-f(x^*)\leq\bar f(x)-\bar f(x^*)+\varepsilon/2
\leq\delta+\varepsilon/2\leq\varepsilon.
\]
If $x$ lies outside this sublevel set, the segment from $x^*$ to $x$ contains a point $z$ with $f(z)-f(x^*)=2\varepsilon$. Convexity of $\bar f$ gives $\bar f(z)\leq\bar f(x^*)+\delta$, and the preceding inequality at $z$ is a contradiction. Finally add the solver's failure probability.
\end{proof}

Uniqueness is relevant to this particular offline bound: distance to a set of minimizers does not control the diameter of a sublevel set. For example, $f(u,v)=|u|$ on $[-1,1]^2$ has sharp growth toward $\{0\}\times[-1,1]$, while its small sublevel sets retain diameter at least $2$. With multiple minimizers, a covering argument must account for the geometry of the solution set as well.

\begin{example}[Sample complexity under $s$-growth]
Consider $p=2$, $Q=B_2^n(1)$ and
\[
f(x,\xi)=\|x\|_2^s-s\langle\xi,x\rangle,
\qquad \xi\sim\mathcal N(0,\sigma^2I_n),\quad s\geq1.
\]
Then $f(x)=\|x\|_2^s$, $x^*=0$, and $\mu_{2,s}=1$. The centered increment is $-s\langle\xi,y-x\rangle$, so \eqref{eq:var} holds with $\lambda^2=s^2\sigma^2$. Moreover,
\[
\bar f(x)=\|x\|_2^s-s\langle\bar\xi_N,x\rangle,
\qquad \bar\xi_N=\frac1N\sum_{k=1}^N\xi^k\sim\mathcal N(0,\sigma^2N^{-1}I_n).
\]
For $s>1$ and $\bar\xi_N\ne0$, the empirical minimizer is
\[
\hat x^N=\min\{1,\|\bar\xi_N\|_2^{1/(s-1)}\}
\frac{\bar\xi_N}{\|\bar\xi_N\|_2};
\]
it is zero when $\bar\xi_N=0$. Consequently, for $0<\varepsilon<1$,
\[
f(\hat x^N)-f(x^*)\leq\varepsilon
\quad\Longleftrightarrow\quad
\|\bar\xi_N\|_2^2\leq\varepsilon^{2(s-1)/s}.
\]
If $q_{\chi_n^2}(1-\beta)$ denotes the $(1-\beta)$-quantile of a chi-squared variable with $n$ degrees of freedom, the exact sample threshold is
\begin{equation}\label{eq:low_boun_r}
N\geq\frac{\sigma^2 q_{\chi_n^2}(1-\beta)}{\varepsilon^{2(s-1)/s}}.
\end{equation}
At fixed confidence this is of order $n\sigma^2/\varepsilon^{2(s-1)/s}$, showing that the dimension and accuracy dependence in \eqref{eq:N_ERM_r} cannot in general be improved for ERM, with $s>1$ fixed. At $s=1$, the minimizer is zero when $\|\bar\xi_N\|_2<1$ and lies on the unit sphere when $\|\bar\xi_N\|_2>1$. The equality event has probability zero for nondegenerate Gaussian noise. Thus the threshold for any $\varepsilon<1$ is $N\geq\sigma^2q_{\chi_n^2}(1-\beta)$, independent of $\varepsilon$.

There is no contradiction with the dimension-independent strongly convex bound. Here the Lipschitz constant is random:
\begin{equation}\label{eq:Lip_p_stoch}
|f(y,\xi)-f(x,\xi)|\leq M_p(\xi)\|y-x\|_p.
\end{equation}
For this example its exact value on the unit ball is $M_2(\xi)=s(1+\|\xi\|_2)$, and
\[
\EE M_2(\xi)^2=s^2\left(1+2\EE\|\xi\|_2+n\sigma^2\right).
\]
The uniform bound in Theorem~\ref{Th:StrCnv} is absent, and even the second-moment scale contains the dimension.
\end{example}

The online counterpart uses growth to localize successive runs of Mirror Descent. This restart principle is developed for uniformly convex problems in \cite{juditsky2014deterministic} and for local growth conditions in \cite{xu2017growth}.
\begin{theorem}\label{Th:r_online}
Assume convexity and the bounded stochastic subgradient oracle of Theorem~\ref{Th:Agarwal}, with $1\leq p\leq2$. Suppose \eqref{eq:r} holds on all of $Q$, and a bound $R_0>0$ with $\operatorname{dist}_p(x^1,S)\leq R_0$ is known. For $0<\varepsilon<\mu_{p,s}R_0^s$, set
\[
\kappa=\left\lceil\log_2\frac{\mu_{p,s}R_0^s}{\varepsilon}\right\rceil.
\]
A restarted Mirror Descent method using fresh samples at every stage satisfies \eqref{eq:P_tilde} with
\[
N=\begin{cases}
\displaystyle\cO_s\left(\kappa+
\frac{K_pM_p^2}{\mu_{p,s}^{2/s}\varepsilon^{2(s-1)/s}}
\log\frac{2\kappa}{\beta}\right),&s>1,\\[2mm]
\displaystyle\cO\left(\kappa+
\kappa\frac{K_pM_p^2}{\mu_{p,1}^2}\log\frac{2\kappa}{\beta}\right),&s=1.
\end{cases}
\]
The constant in $\cO_s$ may depend on $s$. For fixed $s>1$, the sample complexity simplifies to
\[
N(\varepsilon,\beta)=\widetilde{\cO}\left(
\frac{K_pM_p^2}{\mu_{p,s}^{2/s}\varepsilon^{2(s-1)/s}}\right).
\]
At $s=1$, the number of stages remains an explicit logarithmic factor.
\end{theorem}
\begin{proof}
For clarity, consider $p=2$. At stage $l=0,\ldots,\kappa-1$, restrict the method to the intersection of $Q$ and the ball of known radius $R_l=R_0\,2^{-l/s}$ about the current center $x_l$, where $x_0=x^1$. On successful preceding stages this ball contains a population minimizer. Apply the convex bound with fresh samples, confidence $1-\beta/\kappa$, and target error $\mu_{2,s}R_l^s/2$. It suffices to take
\[
N_l=\left\lceil\frac{CM_2^2}{\mu_{2,s}^2R_l^{2(s-1)}}
\log\frac{2\kappa}{\beta}\right\rceil.
\]
The average $x_{l+1}$ produced at this stage satisfies
\[
f(x_{l+1})-f^*\leq\frac{\mu_{2,s}R_l^s}{2},\qquad
\operatorname{dist}_2(x_{l+1},S)\leq2^{-1/s}R_l=R_{l+1}.
\]
Thus the induction uses known upper bounds on distances, and the last stage has error at most $\mu_{2,s}R_0^s2^{-\kappa}\leq\varepsilon$. A conditional union bound gives total failure probability at most $\beta$.

For $s>1$, the main terms of $\sum_{l=0}^{\kappa-1}N_l$ form a geometric series with ratio $2^{2(s-1)/s}>1$, so its sum is bounded by a constant depending on $s$ times the last term, plus $\kappa$ for rounding. For $s=1$, all stages cost the same order and their sum contains the factor $\kappa$. Rescaling the mirror maps from Theorem~\ref{Th:Agarwal} to the successive $p$-balls gives the factor $K_p$ in the general case.
\end{proof}

\section{Concluding remarks}
We collect here several comments that clarify the comparison between online and offline approaches and connect it with more recent developments. The central distinction remains the same: statistical accuracy, the number of oracle calls, and the number of sequential or communication rounds are different resources. Additional structure can improve one of them without improving the others.

\subsection{Weakening of uniform Lipschitz condition in online approach}
An important remark concerning the online approach is that uniform Lipschitz continuity in \eqref{eq:Lip_p} can often be replaced by a moment condition. If the stochastic subgradient is conditionally unbiased and its squared dual norm has a bounded expectation, stochastic mirror descent retains the usual convergence order in expectation \cite{nemirovski2009robust,juditsky2014deterministic}. In particular, it is sufficient that the random Lipschitz constant in \eqref{eq:Lip_p_stoch} satisfy $\EE[M_p(\xi)^2]<\infty$, with the corresponding choice of subgradient.

High-probability guarantees require more care. A finite second moment alone does not give the usual logarithmic dependence on the failure probability for unmodified SGD. Robust aggregation or gradient clipping can recover such guarantees under suitable assumptions; see \cite{nazin2019algorithms,gorbunov2021near}. More recent work gives refined, nearly sub-Gaussian statistical guarantees for clipped SGD in streaming estimation, distinguishing the trace and the operator norm of the noise covariance \cite{das2024streaming}. This is particularly relevant when the data cannot be stored and repeatedly processed.

Clipping nevertheless has a cost. The finite-sample comparison in \cite{suvorikova2026clipped} separates the Gaussian and heavy-tail contributions for averaged SGD. For conditionally centered noise $\zeta_k$ with conditional moments bounded by $\sigma_2^2$ and $\sigma_r^r$, $r>2$, its basic concentration estimate is
\[
\left\|\frac1N\sum_{k=1}^N\zeta_k\right\|_2
\le C\sigma_2\sqrt{\frac{\ell_\rho}{N}}
 +C_r\sigma_r\rho^{-1/r}N^{-1+1/r},
\qquad \ell_\rho=\log(4/\rho),
\]
with probability at least $1-\rho$, where $0<\rho<1$ and $C_r$ depends only on $r$. Balanced clipping replaces the polynomial confidence dependence of the correction by a logarithmic one. For whole-gradient clipping, however, that correction involves the moment of the full stochastic gradient, including its signal. The projected-convex and quadratic averaging bounds also retain optimization and initialization terms. Once the Gaussian term dominates, clipping need not improve the leading order. At $r=2$, its confidence improvement can affect the leading term itself. A fixed threshold is a different choice: it increases asymptotic variance on a scalar Gaussian quadratic and can shift the limiting point under asymmetric noise. Thus robust high-confidence estimation and efficient ordinary averaging call for different tuning.

With only a moment of order $1+\alpha$, $0<\alpha\le1$, the worst-case dependence on accuracy deteriorates. For example, in the Euclidean convex setting with a bounded feasible radius and bounded $(1+\alpha)$-moment of the stochastic subgradient, appropriate mirror-descent methods have the orders
\[
\EE[f(\bar x^N)-f(x^*)]=\cO\bigl(N^{-\alpha/(1+\alpha)}\bigr),
\qquad N=\cO\bigl(\varepsilon^{-(1+\alpha)/\alpha}\bigr),
\]
for expected error and sample complexity, with other parameters fixed \cite{vural2022mirror}. Other norms require a compatible uniformly convex prox geometry; the moment is measured in the dual norm, and its interaction with this geometry matters. Adaptation to unknown tail parameters is a separate issue \cite{zhang2022parameter}.

Additional structure of the noise can change this picture. Under the density and symmetry conditions in \cite{pmlr-v238-puchkin24a}, smoothed median-of-means estimates combined with clipping recover faster rates even for very heavy tails. Symmetry alone is not a general guarantee for ordinary SGD. Related median-clipping methods exploit symmetric heavy-tailed noise in two-point function-value oracles \cite{kornilov2024median}; the oracle distinctions are discussed in Section~\ref{Sec:zeroth_order}. For the offline approach, analogous questions concern robust empirical objectives and concentration of their minimizers; see \cite{shapiro2021lectures}. Thus heavy-tailed data need not force a choice between the two approaches, but they do affect the appropriate estimator and algorithm.

\subsection{Weakening of the convexity condition}
The online analysis in expectation often needs convexity only of the population objective $f(x)=\EE f(x,\xi)$, together with an unbiased stochastic subgradient and the required moment bound. In contrast, convexity of the individual losses is used in many dimension-independent guarantees for empirical minimization. The distinction is substantive: \cite{sekhari2021sgd} constructs problems with a convex population objective for which SGD learns, whereas empirical minimization with a regularizer fixed independently of the sample does not have the same guarantee. This statement does not cover arbitrary data-dependent regularizers.

For general non-convex losses, a uniform-convergence argument on a bounded finite-dimensional domain may still give a sample bound of order $n\varepsilon^{-2}$ up to logarithmic and scale factors, as discussed in Section~\ref{Sec:NonConv_Conv}. More refined complexity measures can replace the ambient dimension in structured statistical models \cite{shalev2014understanding,rakhlin2014statistical,bach2021learning}. This is a statistical conclusion. Finding a global minimum of the resulting empirical objective can remain difficult, and a covering argument does not by itself give a gradient-oracle complexity for that task. Conversely, the usual non-convex SGD guarantees concern stationarity, which is a different target from global excess risk.

There are useful intermediate classes between convex and unrestricted non-convex optimization. Error bounds and growth conditions can support restarts, while the Polyak--Lojasiewicz condition relates function error to the gradient norm \cite{karimi2016linear,necoara2019linear,belkin_2021}. With smoothness, the latter gives linear convergence of deterministic gradient descent and stochastic convergence to a noise-dependent neighborhood; a biased stochastic oracle requires an additional bias term \cite{ajalloeian2020convergence}. Quasi-convexity, growth, and the Polyak--Lojasiewicz condition are different assumptions, and their statistical consequences need not coincide. For the offline approach one must control both the geometry of the empirical problem and its relation to the population problem.

\subsection{How to make the online approach adaptive?}
To answer this question, we return to SGD \eqref{eq:projSGD},
\[
x^{k+1}=\pi_Q\bigl(x^k-\gamma_k g_k\bigr),
\qquad g_k=\nabla_x f(x^k,\xi^k).
\]
The constant choice $\gamma_k=R_2/(M_2\sqrt{N})$ requires the horizon and bounds on the gradient magnitude and distance. On a bounded set, a decreasing stepsize of order $R_2/(M_2\sqrt{k})$ removes knowledge of the horizon. The precise averaging and proof determine whether a logarithmic factor appears; horizon-free regret bounds of order $\sqrt{N}$ are available \cite{nemirovski2009robust,nesterov2009primal}.

AdaGrad replaces a prescribed gradient bound by the observed gradients \cite{duchi2011adaptive,duchi2018introductory}. A scalar version uses
\[
\gamma_k=\frac{R_2}{\sqrt{a^2+\sum_{j=1}^k\|g_j\|_2^2}},\qquad a>0,
\]
where $R_2$ is an appropriate known radius or diameter bound and $a$ prevents a zero denominator. Its analysis uses an adaptive regret inequality: since $\gamma_k$ contains the current gradient, it should not be treated as independent of the current noise. Coordinatewise variants exploit anisotropy and can improve on a scalar bound on particular problems. Adam and its variants build on related adaptive ideas \cite{kingma2014adam,reddi2019convergence}, but their convergence depends on the precise update and assumptions.

Adaptation is useful even when it does not improve the worst-case order over a well-tuned method \cite{bach2019universal,ene2021adaptive}. Recent progress addresses several kinds of missing information separately. DoG \cite{pmlr-v202-ivgi23a} and D-adaptation \cite{defazio2023dadaptation} estimate a distance scale from the trajectory. The accelerated U-DoG method attains nearly optimal smooth stochastic convex rates with only loose distance and noise bounds, and gives high-probability guarantees under sub-Gaussian noise \cite{pmlr-v247-kreisler24a}. These qualifications are part of the meaning of parameter-free in this setting.

A different objective is to track smoothness as well as noise. Universal methods use gradient and iterate differences to adapt to an unknown H\"older smoothness exponent, its constant, and the stochastic variance \cite{gasnikov2018universal,rodomanov2024universalgradientmethodsstochastic}. Thus one algorithm can interpolate between smooth and non-smooth regimes without accumulating only gradient magnitudes.

Unbounded gradients also motivate normalized and clipped updates. Under generalized Lipschitz conditions, where subgradient norms grow with the function gap, \cite{kovalev2026stochastic} establishes improved guarantees for AdamW with clipped updates and explains a role of exponential averaging. This gives a theoretical reason to study these mechanisms beyond uniformly bounded gradients. It does not amount to a universal comparison of optimizers for non-convex neural-network training.

\subsection{Overparametrization}\label{Sec:overparam}
Consider the unconstrained problem
\[
\min_{x\in\R^n}f(x),\qquad f(x)=\EE_\xi f(x,\xi),
\]
where $f$ is $\mu$-strongly convex and, almost surely, each loss is convex with an $L$-Lipschitz gradient:
\begin{equation}\label{eq:over_L}
\|\nabla_x f(y,\xi)-\nabla_x f(x,\xi)\|_2\le L\|y-x\|_2.
\end{equation}
For constant-step SGD, $x^{k+1}=x^k-\gamma\nabla_x f(x^k,\xi^k)$, the noise at the solution is more informative than a uniform bound on all gradient norms \cite{moulines2011non}. Put
\[
\sigma_*^2=\EE_\xi\|\nabla_x f(x^*,\xi)\|_2^2,
\]
using $\nabla f(x^*)=0$. For $0<\gamma\le1/(2L)$,
\begin{equation}\label{eq:over_converg}
\EE\|x^{N+1}-x^*\|_2^2
\le (1-\gamma\mu)^N\|x^1-x^*\|_2^2
       +\frac{2\gamma\sigma_*^2}{\mu}.
\end{equation}

The short proof also identifies the needed assumptions. Smoothness and convexity of each realization give
\begin{align*}
\EE_\xi\|\nabla_x f(x,\xi)\|_2^2
&\le 2\EE_\xi\|\nabla_x f(x,\xi)-\nabla_x f(x^*,\xi)\|_2^2+2\sigma_*^2\\
&\le4L\bigl(f(x)-f(x^*)\bigr)+2\sigma_*^2.
\end{align*}
Expanding the square of one SGD step and using strong convexity yields
\begin{align*}
\EE[\|x^{k+1}-x^*\|_2^2\mid x^k]
&\le(1-\gamma\mu)\|x^k-x^*\|_2^2\\
&\quad-2\gamma(1-2L\gamma)\bigl(f(x^k)-f(x^*)\bigr)
       +2\gamma^2\sigma_*^2.
\end{align*}
The middle term is non-positive. Iterating the resulting recursion proves \eqref{eq:over_converg}. An expected-smoothness condition can replace realization-wise convexity and smoothness, but smoothness of non-convex realizations alone does not justify this proof.

The interpolation effect appears when $\sigma_*^2=0$: each sampled loss has zero gradient at the common solution, and the constant-step bound becomes geometric. If $\sigma_*^2$ is small but positive, there is instead a small noise floor. The distinction matters when discussing very high accuracy.

For a fixed empirical problem we can write
\begin{equation}\label{eq:sum_rand}
\min_{x\in\R^n}\left[\bar f(x)=\frac1N\sum_{i=1}^Nf(x,\xi^i)
=\EE_I f(x,\xi^I)\right],
\end{equation}
where $I$ is uniform on $\{1,\ldots,N\}$. Its relevant noise level is
\[
\widehat\sigma_*^2=\frac1N\sum_{i=1}^N
\|\nabla_x f(\hat x^N,\xi^i)\|_2^2,
\]
with the empirical minimizer $\hat x^N$, rather than the population minimizer. For squared losses, a consistent interpolation system makes this quantity zero. Having more parameters than observations can make consistency possible, but does not guarantee it. Moreover, an underdetermined interpolating least-squares objective is not strongly convex on the entire parameter space; a growth or Polyak--Lojasiewicz condition, or distance to the solution set, is then the appropriate language.

This observation explains fast optimization of some overparameterized empirical problems, but does not compare their population risk with that of an online algorithm. Optimization error and statistical error still have to be separated. Acceleration and minibatching under interpolation or state-dependent noise are developed in \cite{woodworth2021even,ilandarideva2025accelerated}. The latter also treats non-Euclidean geometry. Such results make the noise level depend on progress toward the solution instead of imposing a fixed worst-case variance everywhere.

\subsection{Acceleration and batching for smooth convex optimization problems in online approach}\label{Sec:Acc_Batch}
Consider
\begin{equation}\label{eq:F}
\min_{x\in Q}f(x),
\end{equation}
where $f$ is convex, $Q$ is closed and convex, a solution exists, and, for $L>0$,
\begin{equation}\label{eq:LipF}
\|\nabla f(x)-\nabla f(y)\|_2\le L\|x-y\|_2.
\end{equation}
With exact gradients, accelerated methods obtain function error of order $LR^2/N^2$, where $R=\|x^1-x^*\|_2$ \cite{nesterov2018lectures,lan2020first,lin2020accelerated}. In the stochastic setting, batching can preserve the small number of sequential iterations while controlling the total number of samples.

For clarity, let a fresh minibatch at the current point define
\[
g_r(x)=\frac1r\sum_{j=1}^r\nabla_x f(x,\xi^j),\qquad
 e(x)=g_r(x)-\nabla f(x),
\]
where $\EE[e(x)\mid x]=0$ and $\EE[\|e(x)\|_2^2\mid x]\le\sigma^2/r$.
Convexity and smoothness imply the following inexact model \cite{devolder2013exactness,dvinskikh2020accelerated,Dvinskikh2021}:
\begin{equation}\label{eq:delta}
-\delta_1(x,y)\le f(y)-f(x)-\langle g_r(x),y-x\rangle
\le L\|y-x\|_2^2+\delta_2(x),
\end{equation}
where
\[
\delta_1(x,y)=\langle e(x),y-x\rangle,\qquad
\delta_2(x)=\frac{\|e(x)\|_2^2}{2L},\qquad
\EE\delta_2(x)\le\delta:=\frac{\sigma^2}{2Lr}.
\]
The upper bound uses Young's inequality and has model constant $2L$. The lower error has zero conditional expectation when $x$ and $y$ are fixed before the batch is sampled. In particular, this applies to a current point and a fixed optimum. It need not apply to the next iterate, which uses the same batch; the upper quadratic estimate is designed to handle precisely that dependence.

Suppose the analysis of an inexact-oracle algorithm gives
\begin{equation}\label{RC}
\EE[f(x^N)-f(x^*)]
=\cO\left(\frac{LR^2}{N^\alpha}+N^\zeta\delta\right),
\end{equation}
with the required predictable uses of the lower model. Then balancing its terms suggests
\[
N\asymp\max\left\{1,\left(\frac{LR^2}{\varepsilon}\right)^{1/\alpha}\right\},
\qquad
r\asymp\max\left\{1,\frac{\sigma^2N^\zeta}{L\varepsilon}\right\},
\]
with integer rounding. The familiar choices are $(\alpha,\zeta)=(1,0)$ for gradient descent and $(2,1)$ for accelerated schemes in this model. For $0<\varepsilon\le LR^2$, acceleration gives
\[
N=\cO\left(1+\sqrt{\frac{LR^2}{\varepsilon}}\right),\qquad
T=\cO\left(1+\sqrt{\frac{LR^2}{\varepsilon}}
                 +\frac{\sigma^2R^2}{\varepsilon^2}\right)
\]
sequential iterations and stochastic gradient calls, respectively. When the initial gap is already below $\varepsilon$, no optimization is needed. A known upper bound on $R$ can be used for tuning; high-probability localization requires the corresponding additional argument.

Parallel workers can evaluate the batch simultaneously. Thus the statistical term in the total work need not determine the sequential depth. The smooth deterministic term remains present even when $\sigma=0$. These are bounds for the same convex smooth problem class; strongly convex rates should be compared separately. Lower bounds in the relevant parallel oracle models explain the role of acceleration \cite{woodworth2018graph}, while interpolation permits more refined tradeoffs \cite{woodworth2021even}.

The broader lesson is that sensitivity to an inexact oracle can turn a deterministic algorithm into a useful stochastic one. Applications include gradient-free optimization \cite{gasnikov2022power} and distributed saddle-point problems \cite{metelev2022decentralized}. For bilinearly coupled convex-concave problems, stochastic sliding further separates the costs of the two function oracles and the coupling operator under individual H\"older conditions \cite{nguyen2026sliding}; the gap guarantee is discussed in Section~\ref{Sec:SaddleVI}. Non-Euclidean and restarted schemes provide other extensions \cite{dvurechensky2016stochastic,gorbunov2020stochastic}. Higher derivatives require their own error analysis; we return to this point in Section~\ref{Sec:Acc_Tens}.

Momentum can shorten the transient period before averaging reaches its statistical limit \cite{polyakmomentum2026}. Consider a strongly convex quadratic with Hessian $H$, smallest and largest eigenvalues $\mu$ and $L$, and independent centered additive gradient noise of covariance $\Sigma$. Put $\kappa=L/\mu$ and use one sample per update. Compare SGD and Nesterov acceleration with step $1/L$ and, for the latter, momentum $(\sqrt\kappa-1)/(\sqrt\kappa+1)$. For both methods, the full average of the gradient-query points has the same limiting covariance $H^{-1}\Sigma H^{-1}$ after its error is multiplied by $\sqrt m$, where $m$ is the number of updates.

The difference is how quickly this statistical regime is reached. Let $v$ be a unit eigenvector with $Hv=\mu v$ and $v^T\Sigma v>0$. Assume a deterministic start with zero initial velocity, and keep the projected initial displacement $\langle v,x^0-x^*\rangle$, measured in units of $\sqrt{v^T\Sigma v}/\mu$, fixed and nonzero. Along this least-curved direction, the initialization bias becomes negligible compared with the statistical standard error, and the relative variance error vanishes, precisely when
\[
 m\gg\kappa^2\quad\hbox{for SGD},\qquad
 m\gg\kappa\quad\hbox{for Nesterov acceleration}.
\]
These thresholds concern full averaging from the stated start; burn-in or tail averaging can change them. For nonlinear objectives, Richardson--Romberg extrapolation addresses a further issue, the stationary bias of constant-step SGD \cite{sheshukova2025richardson}.

\subsection{Stochastic optimization under generalized smoothness}\label{Sec:GenSmooth}
Uniform smoothness can also be relaxed when curvature grows away from a solution. For a convex differentiable function on $\R^n$, one such model is
\[
\|\nabla f(y)-\nabla f(x)\|_2
\le\bigl(L_0+L_1\|\nabla f(x)\|_2\bigr)\|y-x\|_2,
\qquad \|y-x\|_2\le1/L_1,
\]
where $L_0>0$, $L_1\ge0$, and the distance restriction disappears for $L_1=0$. Momentum makes localization important because an extrapolated point can enter a region of much larger curvature. Under a conditionally unbiased oracle with additive norm-sub-Gaussian noise of scale $\sigma$, ARC-SG stabilizes the gradient and then accelerates within certified smoothness balls \cite{dvinskikh2026localize}. Given a distance bound $R_0\ge\|x^0-x^*\|_2$, put $\bar R=3R_0$ and $\chi=1+L_1\bar R$. Its high-probability oracle bound contains the accuracy-dependent terms
\[
\widetilde{\cO}\left(
 \chi^{1/2}\sqrt{\frac{L_0\bar R^2}{\varepsilon}}
 +\frac{\chi^3\sigma^2\bar R^2}{\varepsilon^2}\right),
\]
in addition to initialization terms independent of $\varepsilon$. The bound uses known parameter and distance bounds; the initialization cost also depends on the chosen initial-gap certificate. Strongly convex restarts recover a statistical term of order $\sigma^2/(\mu\varepsilon)$ at sufficiently small accuracy. The smooth accuracy exponents survive, while the necessity of the displayed generalized-smoothness factors remains open.

A complementary model lets curvature grow with the function gap $\Delta(x)=f(x)-f(x^*)$:
\[
f(y)\le f(x)+\langle\nabla f(x),y-x\rangle
 +(H_0+H_1\Delta(x))\|y-x\|_2^2.
\]
Here $H_0,H_1\ge0$, and the model is required for
$\|y-x\|_2\le[(2\sqrt3+\sqrt6)\sqrt{H_1}]^{-1}$, interpreted as unrestricted when $H_1=0$. For a conditionally unbiased fresh gradient, \cite{lobanov2026acceleratedheavy} permits nonnegative constants $\sigma_0,\sigma_1,\sigma_2$ such that
\[
\EE[\|g(x,\xi)-\nabla f(x)\|_2^r\mid x]
\le\sigma_0^r+\sigma_1^r\|\nabla f(x)\|_2^r
                    +\sigma_2^r\Delta(x)^{r/2},
\qquad 1<r\le2.
\]
Clipping, projection, and gap-halving restarts retain accelerated square-root curvature dependence using one sample per iteration. With known parameters and an initial-gap bound, the convex guarantee on a bounded localization set of diameter $D$ has an additive-noise contribution
\[
\widetilde{\cO}\left(\left(\frac{\sigma_0D}{\varepsilon}\right)^s\right),
\qquad s=\frac{r}{r-1}.
\]
The gradient- and gap-dependent noise contributions have at most the weaker accuracy power $\varepsilon^{-s/2}$, with other parameters fixed. Under $\mu$-strong convexity, the additive contribution becomes
\[
\widetilde{\cO}\left(\left(\frac{\sigma_0}{\sqrt{\mu\varepsilon}}\right)^s\right);
\]
when $\sigma_0=0$, the entire accuracy dependence is logarithmic. Thus heavy tails and acceleration are compatible, and vanishing noise can be as important as its tail exponent.

\subsection{Sum-type problems and offline approach}\label{Sec:VR_short}
At the end of the offline approach, we need to solve an empirical problem,
\begin{equation}\label{eq:ERM_SC}
\min_{x\in Q}\left[\bar f(x)=\frac1N\sum_{i=1}^Nf(x,\xi^i)\right].
\end{equation}
For the strongly convex Lipschitz setting of Theorem~\ref{Th:StrCnv}, a sample size of order $M^2/(\mu\varepsilon)$, up to confidence factors, is sufficient. The empirical optimization accuracy may have to be considerably smaller than $\varepsilon$; the distance-based transfer used there leads to a scale $\delta\asymp\mu\varepsilon^2/M^2$. A plain non-smooth stochastic solver of the empirical problem can therefore use many more component calls than the sample size. Smoothness alone does not remove this issue if the finite-sum structure is ignored.

Variance reduction exploits repeated access to the same losses. When the components are convex and $L$-smooth and the average is $\mu$-strongly convex, accelerated VR methods reach expected empirical error $\delta$ with component-gradient complexity
\begin{equation}\label{eq:VR_rate}
\cO\left(\left(N+\sqrt{\frac{NL}{\mu}}\right)
             \log\left(2+\frac{\Delta f}{\delta}\right)\right),
\end{equation}
where $\Delta f$ bounds the initial empirical gap \cite{woodworth2016tight,allen2016optimal,lan2020first,lin2020accelerated,kovalev2020dont}. The characteristic dependence on $N$ and $L/\mu$ matches the appropriate incremental-oracle lower bounds in sufficiently high dimension. If $L/\mu\lesssim N$, the cost is nearly linear in the sample size. Importance sampling and expected-smoothness analyses can replace the worst component constant by more favorable problem-dependent quantities \cite{JMLR:v22:20-156}.

The underlying idea is a control variate. For a snapshot $w$ with an available full gradient and an index $i$ sampled uniformly from $\{1,\ldots,N\}$, use
\[
v_i(x,w)=\nabla_x f(x,\xi^i)-\nabla_x f(w,\xi^i)+\nabla\bar f(w),
\qquad \EE_i v_i(x,w)=\nabla\bar f(x).
\]
The snapshot is refreshed periodically or randomly. For an unconstrained empirical minimizer $\hat x^N$, convex $L$-smooth components give
\begin{equation}\label{eq:VR_var}
\EE_i\|v_i(x,w)\|_2^2
\le4L\left(\bar f(x)-\bar f(\hat x^N)
                 +\bar f(w)-\bar f(\hat x^N)\right).
\end{equation}
Both the current point and the snapshot must approach optimality. With constraints, the raw gradient need not vanish at the solution; one instead controls the variance or centers the estimate at $\nabla\bar f(\hat x^N)$, and uses the corresponding Bregman gaps. The same principle extends to proximal VR algorithms.

Some non-smooth problems, including structured support-vector formulations, admit efficient smoothing reductions \cite{allen2016optimal}. The smoothing bias must, however, fit the intended error budget. To solve the original empirical objective to accuracy $\delta$, a smoothing error of order $\varepsilon$ is generally insufficient when $\delta\ll\varepsilon$. Alternatively, one can smooth the statistical problem at scale $\varepsilon$ and analyze that modified problem directly. The cost of obtaining a smoothed gradient also depends on the available structure or proximal oracle.

Variance reduction is now useful beyond first-order methods. Under average second-order smoothness, it reduces the number of component gradient and Hessian evaluations as well; the recent bound in Section~\ref{Sec:Acc_Tens} makes this improvement explicit. This connects two advantages of the offline approach: stored data can support both control variates and more informative local models.

\subsection{Composite optimization}
Regularizers play an important role in empirical minimization. A non-smooth regularizer, such as $\|x\|_1$ in LASSO, does not necessarily force us to treat the whole objective as a generic non-smooth function. For
\begin{equation}\label{eq:SOP_}
\min_{x\in Q}\left[\EE_\xi f(x,\xi)+h(x)\right],
\end{equation}
with a proper closed convex function $h$ and a well-defined proximal subproblem, the proximal stochastic-gradient step is
\[
x^{k+1}=\argmin_{x\in Q}\left\{\langle g_k,x-x^k\rangle
           +\frac{\|x-x^k\|_2^2}{2\gamma_k}+h(x)\right\}.
\]
Projection is the case $h=0$. If the proximal subproblem is inexpensive, the regularizer is called proximal-friendly. For the $\ell_1$ norm without other constraints this operation is soft thresholding. Under the corresponding smoothness or moment assumptions on the stochastic part, proximal methods retain the relevant stochastic convergence orders; $h$ itself need not have a Lipschitz gradient \cite{duchi2010composite,nesterov2013gradient,lan2020first}. Stochastic dual averaging is a related, distinct construction \cite{xiao2009dual}.

The same separation is compatible with acceleration, batching, and variance reduction \cite{beck2009fast,dubois2022fast,lan2020first}. Universal stochastic methods also allow a simple convex composite term while adapting to the regularity of the remaining objective \cite{rodomanov2024universalgradientmethodsstochastic}. Thus adaptivity need not discard useful regularization structure.

When a proximal step is expensive but separate oracles for the terms are available, sliding methods can allocate different numbers of calls to each term \cite{lan2020first,kovalev2022sliding}. This is useful in distributed optimization, where one oracle may require communication. More general convex model frameworks include composite objectives and maxima $f(x)=\max\{f_1(x),\ldots,f_m(x)\}$ of convex functions \cite{nemirovski1985optimal,stonyakin2021inexact}. Stochastic extensions require control of the corresponding model errors \cite{dvinskikh2020accelerated}.

\subsection{Overfitting and early stopping for offline approach}
Regularization is one way to align the sample complexities of online and offline approaches. Another is to use the empirical data only once: SGD on independent observations, without revisiting them, has exactly the same distribution as online SGD. Subsequent passes reduce empirical optimization error but reuse information already present in the sample. Their effect on population risk depends on the stability and implicit regularization of the algorithm.

A useful decomposition is
\[
\EE[f(x_T)-f(x^*)]
\le\EE[f(x_T)-\bar f(x_T)]
   +\EE[\bar f(x_T)-\min_{x\in Q}\bar f(x)].
\]
Here the omitted comparison term is non-positive in expectation, because $x^*$ is fixed independently of the sample. Stability controls the first term, and optimization theory controls the second. The analysis of \cite{hardt2016train} shows how stepsizes and the number of passes enter this tradeoff for SGD. A bound that increases with the horizon does not by itself prove that the actual risk increases.

The same principle applies to full gradient descent. For convex $L$-smooth, $M$-Lipschitz losses on a convex set of diameter $D$, projected GD with $\eta\le1/L$ and an averaged output has the elementary expected-risk bound
\[
\EE[f(\bar x^T)-f(x^*)]
\le\frac{D^2}{2\eta T}+\frac{2M^2\eta T}{N}.
\]
The first term is optimization error; nonexpansiveness of the gradient step on neighboring samples gives the stability term. Choosing $\eta T$ of order $D\sqrt{N}/M$ yields the familiar $N^{-1/2}$ statistical rate. Full GD may require more component evaluations than a single pass of SGD, but the statistical comparison is not an unconditional separation between them.

Negative results must therefore be read with their problem class and resource fixed. The GD/SGD separations in \cite{amir2021never,sekhari2021sgd} concern particular iteration or statistical models. In particular, a lower bound for losses whose expectation is convex cannot automatically be transferred to smooth convex individual losses. Likewise, the refined rates in \cite{lin2016generalization} depend on approximation conditions of the learning model as well as the stepsize schedule.

Strong convexity improves stability and supports an expected statistical scale of order $M^2/(\mu N)$, but does not exclude a benefit from early stopping on an individual distribution. Even for the quadratic loss $f(x,\xi)=\tfrac12(x-\xi)^2$ with $\xi=\pm1$ equally likely, GD started at $c\in[-1,1]$ with $0<\eta<1$ gives
\[
\EE[f(x^t)-f(0)]
=\frac12\left(q_t^2c^2+\frac{(1-q_t)^2}{N}\right),
\qquad q_t=(1-\eta)^t.
\]
The minimum can occur before the empirical minimizer is reached. This is compatible with an order-$N^{-1}$ statistical term plus a decaying optimization error: a minimax rate and the shape of a particular learning curve answer different questions.

These observations also explain why accurately solving a properly regularized empirical problem with a deterministic, VR, or second-order method remains a sound strategy. What matters is that the regularization and optimization accuracy are chosen for the population-risk target, rather than interpreting empirical convergence alone as statistical improvement.

\subsection{Distributed optimization}\label{Sec:distr_1}
The batch in Section~\ref{Sec:Acc_Batch} can be split among several workers. Aggregating their partial gradients requires communication, so reducing sequential iterations also reduces communication rounds in this centralized implementation. Whether local steps can improve the tradeoff depends on the oracle model. In the homogeneous intermittent-communication model of \cite{woodworth2021min}, each worker can make a prescribed number of local stochastic-gradient queries between rounds. The better of minibatch accelerated SGD and single-machine accelerated SGD is minimax optimal up to logarithmic factors. This result concerns a worst-case smooth convex class; it is not a universal ranking of federated-learning procedures with heterogeneous data, partial participation, or different communication constraints \cite{kairouz2021advances}.

Quadratic objectives illustrate how additional structure helps. Their gradients are affine, so the average of the true gradients at local points equals the true gradient at the average point. With independent fresh noises and appropriate linear updates, this permits local computation followed by averaging, including one-shot schemes, with favorable variance reduction across workers \cite{woodworth2020is}. The guarantee concerns the error of the output; it need not identify the full distribution of the local and centralized trajectories. Nearly quadratic objectives and local Newton models are natural settings in which to exploit this observation \cite{woodworth2021stochastic}.

The offline approach gives another source of structure. If $N=ms$ observations are stored on $m$ workers, write
\[
\bar f(x)=\frac1m\sum_{k=1}^m\bar f_k(x),\qquad
\bar f_k(x)=\frac1s\sum_{j=1}^sf(x,\xi^{k,j}).
\]
For convex $L$-smooth component losses and a $\mu$-strongly convex average, a full-gradient accelerated method needs
\begin{align*}
 &\widetilde{\cO}\left(\sqrt{\frac{L}{\mu}}\right)
 &&\hbox{communication rounds},\\
 &\widetilde{\cO}\left(s\sqrt{\frac{L}{\mu}}\right)
 &&\hbox{component calls per worker}.
\end{align*} Here $L$ is a valid smoothness bound; using a common upper bound on component smoothness gives a simple comparison. Distributed VR methods can retain the communication order while reducing the local component work to
\[
\widetilde{\cO}\left(s+\sqrt{\frac{sL}{\mu}}\right)
\]
in the corresponding model \cite{li2020variance,hendrikx2021optimal}.

Statistical similarity addresses communication from a different direction. Suppose the gradients of $\bar f-\bar f_k$ are $l$-Lipschitz, with $l\ll L$. A master worker can use its own local objective as a preconditioner, while communicating only the correction. Accelerated gradient sliding makes the idea precise \cite{kovalev2022sliding}. For
\[
\min_{x\in\R^n}[g(x)+h(x)],
\]
assume that $g$ has $L_g$-Lipschitz gradient, $h$ is convex and $L_h$-smooth, and their sum is $\mu$-strongly convex. The first term need not be convex. A sliding iteration forms $\tilde x^t=\tau x^t+(1-\tau)x_f^t$ and approximately solves
\begin{equation}\label{eq:aux_prob}
 x_f^{t+1}\approx\argmin_{x\in\R^n}\left[A^t(x):=g(\tilde x^t)
 +\langle\nabla g(\tilde x^t),x-\tilde x^t\rangle
 +L_g\|x-\tilde x^t\|_2^2+h(x)\right],
\end{equation}
with the sufficient inner-accuracy condition
\begin{equation}\label{eq:criteria_aux}
\|\nabla A^t(x_f^{t+1})\|_2^2
\le\frac{L_g^2}{3}\left\|\tilde x^t-\argmin_{z\in\R^n} A^t(z)\right\|_2^2.
\end{equation}
The remaining update and parameters are
\[
x^{t+1}=x^t+\eta\mu(x_f^{t+1}-x^t)-\eta\nabla(g+h)(x_f^{t+1}),
\]
\[
\tau=\min\left\{1,\frac{\sqrt\mu}{2\sqrt{L_g}}\right\},\qquad
\eta=\min\left\{\frac1{2\mu},\frac1{2\sqrt{\mu L_g}}\right\}.
\]
With a positive admissible $L_g$, suitable inner solvers give the separate oracle bounds
\[
 N_g=\widetilde{\cO}\left(1+\sqrt{\frac{L_g}{\mu}}\right),\qquad
 N_h=\widetilde{\cO}\left(1+\sqrt{\frac{L_g+L_h}{\mu}}\right),
\]
including the required inner accuracy. The number of outer iterations has the same order as $N_g$. If $h=0$, the inner step is $x_f^{t+1}=\tilde x^t-\nabla g(\tilde x^t)/(2L_g)$. If $g=0$, one still chooses a positive upper bound $L_g$, which supplies the proximal regularization parameter. This connects sliding with Catalyst and its more robust variants \cite{lin2015universal,carmon2022recapp}.

For distributed empirical minimization, use
\begin{equation}\label{eq:simil_trick}
\bar f(x)=\underbrace{\bar f(x)-\bar f_1(x)}_{g(x)}
          +\underbrace{\bar f_1(x)}_{h(x)}.
\end{equation}
The master has direct access to $h$, while evaluating $\nabla g$ requires the workers' gradients. Hence $L_g=l$ gives communication complexity
\[
\widetilde{\cO}\left(1+\sqrt{\frac{l}{\mu}}\right),
\]
consistent with similarity-based lower bounds in the relevant model \cite{arjevani2015communication}. For independent data, Hessian concentration may make $l$ decrease with local sample size, often like $s^{-1/2}$ up to scale, dimension, and confidence factors \cite{hendrikx2020statistically}. Independence alone does not imply a uniform bound over all $x$; boundedness or tail assumptions and localization are needed.

Combining these benefits remains an active direction, but variance reduction and similarity are already compatible in several settings. SVOGS combines client sampling, variance reduction, and optimistic gradient sliding for distributed convex-concave problems, with nearly matching bounds for local computations and communication \cite{zhou2024svogs}. Accelerated methods with compressed messages also exploit similarity through variance reduction or error feedback \cite{bylinkin2025compressed}. Client-level sampling, component sampling within a worker, and compression count different resources; their best combination depends on the data distribution and the relative costs of computation and communication.

Second-order local models offer another possibility. With $d=x-\tilde x^t$, the model
\[
\tilde f_1(x)=\bar f_1(\tilde x^t)
  +\langle\nabla\bar f_1(\tilde x^t),d\rangle
  +\tfrac12\langle\nabla^2\bar f_1(\tilde x^t)d,d\rangle
\]
is convex when $\bar f_1$ is convex. If the Hessian of $\bar f_1$ is $L_2$-Lipschitz, the local Hessian discrepancy is bounded by $L_2\|d\|_2$, in addition to the statistical Hessian discrepancy. This motivates fewer communications near a solution, provided localization and inner accuracy are controlled. The quadratic inner structure can also support local stochastic computations. Decentralized networks add graph-dependent communication costs; see \cite{gorbunov2020recent} for an overview.

\subsection{Stochastic zeroth-order optimization}\label{Sec:zeroth_order}
When gradients are unavailable, the information returned by a value oracle becomes decisive. The distinction between one-point and two-point access, together with different models of value errors, is discussed in \cite[Section~4]{gasnikov2016inexactoracle}; see also \cite{gasnikov2022randomizedgradientfree}. We first specify this information, then compare the resulting complexity bounds. We distinguish the number of value evaluations from the number of sequential rounds, each of which can contain many parallel queries.

For $f(x)=\EE_\xi f(x,\xi)$, a \emph{one-point stochastic oracle} returns $f(x,\xi)$ at a requested point $x$, with a fresh sample $\xi$ for each call; the algorithm cannot evaluate the same realization elsewhere. A \emph{two-point stochastic oracle} instead returns
\[
\bigl(f(x^+,\xi),f(x^-,\xi)\bigr)
\]
at two points chosen before the common fresh sample is drawn. Such a pair costs two scalar evaluations. Thus the terminology concerns reuse of randomness, not the number of evaluations in an algorithmic iteration. A random additive offset shared by a pair cancels in its difference, while randomness in the slope, including multiplicative noise, generally remains.

Consider first two-point access without further value errors. Let $Q$ be compact and convex with Euclidean diameter at most $D$ and an available Euclidean projection, let $f$ be convex, and suppose the sample losses have random Euclidean Lipschitz constants $M_2(\xi)$ satisfying $\EE M_2(\xi)^2\le M_2^2$. These properties and the oracle are required on the neighborhood where queries are made. For a radius $\gamma>0$, let $e$ be uniform on the unit sphere and $u$ uniform on the unit ball, independently of $\xi$. The estimator and ball smoothing are
\begin{equation}\label{eq:zo_pair}
\begin{aligned}
G_\gamma(x;e,\xi)
&=\frac{n}{2\gamma}\bigl(f(x+\gamma e,\xi)-f(x-\gamma e,\xi)\bigr)e,\\
f_\gamma(x)&=\EE_u f(x+\gamma u).
\end{aligned}
\end{equation}
Here $\EE G_\gamma(x;e,\xi)=\nabla f_\gamma(x)$, the second moment is $\cO(nM_2^2)$, the smoothing error is at most $M_2\gamma$, and $\nabla f_\gamma$ is Lipschitz with constant at most $\sqrt n M_2/\gamma$. Combining these facts with batched acceleration, \cite{gasnikov2022power} obtains, for expected objective error $\varepsilon$,
\begin{equation}\label{eq:zo_smoothing_complexity}
N_{\rm seq}=\cO\left(\frac{n^{1/4}M_2D}{\varepsilon}\right),\qquad
T_{\rm val}=\widetilde{\cO}\left(\frac{nM_2^2D^2}{\varepsilon^2}\right).
\end{equation}
Here $0<\varepsilon\le M_2D$, and the neighborhood must accommodate $\gamma=\varepsilon/(2M_2)$. The call bound has the optimal dimension and accuracy orders for this oracle class; the smaller number of rounds exploits parallel evaluations. Non-Euclidean versions are developed in the same work. More recently, \cite{dorn2026adaptivezeroth} obtains accelerated two-point guarantees with adaptive steps, batches and smoothing. With a known diameter and evaluation neighborhood, its method does not require the smoothness, Lipschitz or noise constants; adaptation of the radius uses additional moment bounds for scalar differences. Its version with fixed spatial corruption uses a supplied corruption bound when adapting the radius.

A different one-point model, classical in the work of \cite{polyak1990searchaccuracy} and developed further by \cite{akhavan2020higherorder}, observes
\begin{equation}\label{eq:zo_additive_oracle}
Y_t=f(x_t)+\eta_t,\qquad
\EE\eta_t=0,\qquad \EE\eta_t^2\le\sigma^2,
\end{equation}
where the errors are independent across calls and independent of the information used to select $x_t$. Two such calls at $x\pm\gamma e$ still have independent errors: their contribution to the variance of the symmetric difference estimator is $n^2\sigma^2/(2\gamma^2)$ when both variances equal $\sigma^2$. Consequently the second-moment bound behind \eqref{eq:zo_smoothing_complexity} does not apply. Higher-order differences or smoothing kernels reduce approximation bias, while averaging controls the amplified observation noise. For strongly convex objectives, the classical results identify error of order $T^{-(\nu-1)/\nu}$ after $T$ observations for smoothness order $\nu>2$, with other parameters fixed. Implementations using two noisy values per iteration belong to this independent-measurement model.

For a recent sharp bound with $\sigma>0$, assume $\mu I\preceq\nabla^2f(x)\preceq LI$, $0<\mu<L$, on $\R^n$, and $\|x^*\|_2\le R$. Let the derivative of order $r=\lceil\nu\rceil-1$ be $(\nu-r)$-H\"older in operator norm with constant $L_\nu>0$. Starting at zero, \cite{loginov2026minimaxzeroth} attains expected error $0<\varepsilon<LR^2$ with
\begin{equation}\label{eq:zo_high_smooth}
\begin{split}
T_{\rm val}&=\cO_\nu\left(
\frac{n^2\sigma^2L_\nu^{2/(\nu-1)}}{(\mu\varepsilon)^{\nu/(\nu-1)}}
+n\sqrt{L/\mu}\left[1+\log\frac{LR^2}{\varepsilon}\right]\right),\\
N_{\rm seq}&=\cO_\nu\left(\sqrt{L/\mu}\left[1+\log\frac{LR^2}{\varepsilon}\right]\right).
\end{split}
\end{equation}
The leading term, including its dimension and variance dependence, matches the minimax lower bound at sufficiently small accuracy. \citet{zhang2026sharpzeroth} independently establish the same $n^2\varepsilon^{-\nu/(\nu-1)}$ minimax dependence under independent Gaussian value noise, with the variance, convexity, smoothness, and radius parameters fixed. Biased errors independent of the current random directions are also allowed in \cite{loginov2026minimaxzeroth} under conditional second-moment bounds.

The method reconstructs gradients from higher-order differences along a random orthonormal basis. A fixed perturbation $s(x)$ with Lipschitz constant $M_{\rm err}$ gives bias of order $M_{\rm err}$. Query-dependent errors $|\delta_t|\le\Delta$ give a reconstruction error whose norm is bounded by $\cO_\nu(\sqrt n\Delta/h)$ on every realization. The approximation bias is $\cO_\nu(L_\nu h^{\nu-1})$. Their squared magnitudes divided by $\mu$ determine the objective error. The radius is $h=c'_\nu h_\varepsilon$, where
\[
h_\varepsilon=\left(\frac{\mu\varepsilon}{L_\nu^2}\right)^{1/[2(\nu-1)]}.
\]
The same complexity is sufficient under
\begin{equation}\label{eq:zo_admissible_noise}
M_{\rm err}\le c_\nu\sqrt{\mu\varepsilon},\qquad
\Delta\le c_\nu\frac{\sqrt{\mu\varepsilon}\,h_\varepsilon}{\sqrt n}.
\end{equation}
Positive constants $c'_\nu,c_\nu$, depending only on $\nu$, allocate the accuracy budget among these errors \cite[Theorem~1(c)]{loginov2026minimaxzeroth}. Systematic and adaptive errors cannot generally be removed by larger batches.

Finally, \cite{abronin2026gradientfreeconstraints} extends two-point methods to convex constraints $g_i(x)=\EE g_i(x,\xi)\le0$, $i=1,\ldots,m$. Besides sample Lipschitz moments and a Slater condition, estimating constraint levels requires $\EE\|(g_i(x,\xi)-g_i(x))_{i=1}^m\|_\infty^2\le\sigma_g^2$. With other parameters fixed, expected objective error and expected maximum positive violation at most $\varepsilon$ are obtained in
\[
N_{\rm seq}=\cO(\varepsilon^{-1/2})\quad\text{for smooth data},\qquad
N_{\rm seq}=\cO(n^{1/4}\varepsilon^{-1})\quad\text{for Lipschitz data}.
\]
The smooth guarantees require smooth convex population extensions. Here one oracle call returns the objective and the entire constraint vector at a point. Total calls retain the $\varepsilon^{-2}$ order, with polylogarithmic dependence on $m$. A smooth strongly convex objective allows logarithmically many rounds and $\varepsilon^{-1}$ objective calls, while noisy constraint levels can still require $\varepsilon^{-2}$ calls. Scalar component feedback has a different cost.

\subsection{Accelerated tensor methods}\label{Sec:Acc_Tens}
Tensor methods use higher derivatives to obtain more accurate local models. Cubic regularization \cite{nesterov2006cubic}, accelerated Newton schemes \cite{monteiro2013accelerated}, and their higher-order extensions \cite{gasnikov2019near} provide global guarantees for convex optimization. For a convex function with an $L_2$-Lipschitz Hessian, optimal deterministic second-order methods attain
\[
f(x^T)-f(x^*)=\cO\left(\frac{L_2R^3}{T^{7/2}}\right),
\qquad R=\|x^0-x^*\|_2,
\]
in the corresponding unconstrained derivative-oracle model. Optimal acceleration without a multiplicative line-search overhead is developed in \cite{kovalev2022first,carmon2022optimal}; lower-bound formulations and dimension requirements are discussed in \cite{kornowski2020high,garg2021near}. These results count derivative queries and model steps, rather than treating all linear algebra as free in a running-time comparison.

In the online approach, stochastic derivatives introduce distinct sources of error. The sensitivity analysis in \cite{agafonov2024inexact,agafonovadvancing} separates gradient error, Hessian error, and the deterministic approximation term. It is this separation, rather than batching by itself, that permits meaningful stochastic acceleration.

A recent development is the accelerated stochastic contracting cubic Newton method \cite{agafonov2026matching}. Let $Q$ be a closed convex set of diameter $D$, and let $f$ be convex with an $L_2$-Lipschitz Hessian. Assume conditionally unbiased gradient estimates with error second moment at most $\sigma_1^2$, and symmetric positive semidefinite Hessian estimates with spectral-norm error second moment at most $\sigma_2^2$. The Hessian estimates may be biased. With fresh samples and a separately sampled correction gradient, the method achieves
\begin{equation}\label{eq:stochastic_second_order_optimal}
\EE[f(x^T)-f(x^*)]
=\cO\left(\frac{\sigma_1D}{\sqrt T}
          +\frac{\sigma_2D^2}{T^2}
          +\frac{L_2D^3}{T^{7/2}}\right).
\end{equation}
One iteration uses two stochastic gradient evaluations, one stochastic Hessian evaluation, and constrained model subproblems. The three powers of $T$ match the corresponding lower-bound terms. Consequently, for fixed oracle noise levels, a sufficient iteration count is
\[
\cO\left(1+\frac{\sigma_1^2D^2}{\varepsilon^2}
 +\sqrt{\frac{\sigma_2D^2}{\varepsilon}}
 +\left(\frac{L_2D^3}{\varepsilon}\right)^{2/7}\right).
\]
The compact-domain and derivative-oracle assumptions are essential to this statement; the bound does not include the arithmetic cost of solving the models.

The offline analogue benefits from second-order variance reduction. Consider a convex average $\bar f(x)=N^{-1}\sum_{i=1}^Nf_i(x)$ on $\R^n$, with twice continuously differentiable components satisfying
\[
\frac1N\sum_{i=1}^N
\|\nabla^2f_i(x)-\nabla^2f_i(y)\|_{\mathrm{op}}^2
\le L_2^2\|x-y\|_2^2.
\]
Individual components need not be convex. The method of \cite{tsedenov2026secondorder} reaches expected function error $\varepsilon$ using
\begin{equation}\label{eq:finite_sum_second_order_optimal}
\cO\left(N\log(N+1)+N^{6/7}
             \left(\frac{L_2R^3}{\varepsilon}\right)^{2/7}\right)
\end{equation}
component gradient/Hessian evaluations in expectation, where $R=\|x^0-x^*\|_2$. A full derivative costs $N$ component calls. A matching lower bound without the logarithm holds for the stated stochastic second-order span model with independently and uniformly sampled component indices and a sufficiently large dimension. Thus the accuracy-dependent term is optimal in that model, and the full bound is optimal up to a logarithm. The method combines accelerated regularization with a variance-reduced second-order extragradient inner solver. Its line search uses stored derivatives and adds linear algebra, but no further derivative queries.

These two results clarify complementary roles of second-order information. In \eqref{eq:stochastic_second_order_optimal}, it improves the deterministic and Hessian-error terms while the worst-case gradient-noise term remains statistical. In \eqref{eq:finite_sum_second_order_optimal}, repeated access to stored components makes derivative differences useful and reduces the finite-sum overhead relative to repeated full-Hessian steps. Neither comparison says that forming a Hessian is as cheap as forming a gradient.

The arithmetic balance is often favorable for moderately sized local problems in distributed optimization \cite{dvurechensky2021hyperfast}. For example, in a dense generalized linear model, assembling $s$ rank-one Hessian contributions costs order $sn^2$, while a dense factorization costs order $n^3$. When $s$ is comparable to $n$, both costs have the same scale, which can make an accurate second-order model attractive. Implementable tensor schemes exploit suitable derivative structure and model solvers \cite{nesterov2021implementable}; an arbitrary third-order tensor need not have Newton-like storage or arithmetic cost. This leaves considerable scope for structure-aware, inexact, and distributed implementations even when an oracle complexity is optimal.

\subsection{Saddle-point problems and variational inequalities}\label{Sec:SaddleVI}
The online/offline comparison extends beyond minimization. Stability and empirical approximation for stochastic saddle-point problems are studied in \cite{lei2021stability,zhang2021generalization,dvinskikh2022relations,ozdaglar2022good}, with distributed variance reduction in \cite{kovalev2022VR}. Stochastic extragradient and mirror-prox methods address monotone variational inequalities directly \cite{juditsky2011solving,gorbunov2022stochastic,gorbunov2022clipped,beznosikov2023smooth}. For a differentiable convex-concave function $\Phi(x,y)$, the operator $(\nabla_x\Phi(x,y),-\nabla_y\Phi(x,y))$ is monotone; an arbitrary non-convex--non-concave game need not have this property.

On a bounded domain, the deterministic gap for a general Lipschitz monotone operator decreases at order $1/T$, rather than the accelerated $1/T^2$ function-error rate of smooth convex minimization. This does not exclude acceleration for structured saddle problems, residual criteria, or higher-order oracles. The choice of accuracy measure is especially important: a duality gap, an operator residual, and distance to a solution need different assumptions for conversion.

Variance reduction gives different guarantees for gaps and residuals. The contribution of \cite{pichugin2024method} is to combine an optimistic method, a snapshot control variate, and Bregman proximal steps with batching. Consider $F(z)=N^{-1}\sum_{i=1}^N F_i(z)$ on a compact convex set $Q$ of diameter $D$, where $F$ is monotone and each $F_i$ is Euclidean $L$-Lipschitz. Let $h$ be a convex composite term with an available proximal operator. The Minty gap is
\[
\operatorname{Gap}(z)=\sup_{u\in Q}
\{\langle F(u),z-u\rangle+h(z)-h(u)\}.
\]
For batch size $1\le b\lesssim\sqrt N$, the averaged output attains expected gap at most $\varepsilon$ with
\[
\cO\left(N+\frac{\sqrt N\,LD^2}{\varepsilon}\right)
\quad\hbox{component evaluations}
\]
and
\[
\cO\left(\frac{N}{b}+\frac{\sqrt N\,LD^2}{b\varepsilon}\right)
\quad\hbox{proximal updates}.
\]
Thus batching reduces the number of sequential updates without increasing the total component-oracle order; full snapshot evaluations are included. In non-Euclidean geometry, the estimates retain the specified Bregman radius and smoothness constants, with logarithmic dimension factors in the dual-norm analysis. This is an averaged gap guarantee. Last-iterate norm-residual guarantees are obtained through variance-reduced Halpern iteration \cite{cai2024halpern}.

For the inclusion $0\in F(x)+G(x)$, where $F(x)=N^{-1}\sum_{i=1}^N F_i(x)$ is monotone, its components are mean-square $L$-Lipschitz, and $G$ is maximally monotone with an available resolvent, \cite{zhou2026inclusions} obtains a different accuracy guarantee and removes the logarithmic overhead of earlier residual-based methods. Given bounds on $L$ and the initial distance $R$, the method uses
\[
\cO\left(N+\frac{\sqrt N\,LR}{\varepsilon}\right)
\]
expected component and resolvent evaluations to return $y$ and a certificate $g\in G(y)$ satisfying
\[
\bigl(\EE\|F(y)+g\|_2^2\bigr)^{1/2}\le\varepsilon.
\] A matching lower bound holds in its randomized linear-span component-oracle model for $0<\varepsilon\le LR/2$. Individual components need not be monotone. This is a residual guarantee and should not be read as the same statement as a bounded-domain duality-gap bound.

Structure provides a different improvement. For
\[
\min_{x\in X}\max_{y\in Y}
\bigl[f(x)+\langle y,Bx\rangle-g(y)\bigr],
\]
let $X\subseteq\R^{d_x}$ and $Y\subseteq\R^{d_y}$ be nonempty bounded closed convex sets, let $B\in\R^{d_y\times d_x}$, and let the convex functions have H\"older subgradient parameters $(H_x,\nu_x)$ and $(H_y,\nu_y)$, $0\le\nu_x,\nu_y\le1$. Assume unbiased stochastic function oracles with conditional variances $\sigma_x^2,\sigma_y^2$, and exact products with $B,B^\top$. For a fixed start $z^0$, put $\Omega=\sup_{z\in X\times Y}\|z-z^0\|_2^2$. For expected duality gap $\varepsilon$, recursive sliding \cite{nguyen2026sliding} uses $N_f$ and $N_g$ stochastic subgradient evaluations and $N_B$ coupling products:
\begin{align*}
 N_f&=\cO\left(1+
 \left(\frac{H_x\Omega^{(1+\nu_x)/2}}{\varepsilon}\right)^{2/(1+3\nu_x)}
 +\frac{\sigma_x^2\Omega}{\varepsilon^2}\right),\\
 N_g&=\cO\left(1+
 \left(\frac{H_y\Omega^{(1+\nu_y)/2}}{\varepsilon}\right)^{2/(1+3\nu_y)}
 +\frac{\sigma_y^2\Omega}{\varepsilon^2}\right),\\
 N_B&=\cO\left(1+\frac{\|B\|_{\mathrm{op}}\Omega}{\varepsilon}\right).
\end{align*}
Strong convexity yields faster rates for distance to the saddle point. In the smooth case $\nu_x=\nu_y=1$, suppose that $f$ and $g$ are $\mu_x$- and $\mu_y$-strongly convex, respectively, with $\mu_x,\mu_y>0$, and let $(x^*,y^*)$ be the saddle point. The restarted method guarantees
\[
\EE\bigl[\mu_x\|x-x^*\|_2^2+\mu_y\|y-y^*\|_2^2\bigr]\le\varepsilon
\]
using
\begin{align*}
N_f&=\cO\left(\left(1+\sqrt{\frac{H_x}{\mu_x}}\right)\ell_\varepsilon
                   +\frac{\sigma_x^2}{\mu_x\varepsilon}\right),\\
N_g&=\cO\left(\left(1+\sqrt{\frac{H_y}{\mu_y}}\right)\ell_\varepsilon
                   +\frac{\sigma_y^2}{\mu_y\varepsilon}\right),\\
N_B&=\cO\left(\left(1+\frac{\|B\|_{\mathrm{op}}}{\sqrt{\mu_x\mu_y}}\right)
                     \ell_\varepsilon\right).
\end{align*}
Here $\ell_\varepsilon=1+\log(\Gamma_0/\varepsilon)$ for $0<\varepsilon<\Gamma_0$; on the bounded domain above, $\Gamma_0=7(H_x+H_y)\Omega$ is a sufficient initial-error bound. The proof controls a stronger Lyapunov error consisting of this weighted squared distance and Bregman errors at the saddle point. Setting $\sigma_x=\sigma_y=0$ gives the separate deterministic costs. For fixed H\"older exponents $0\le\nu_i<1$, $i\in\{x,y\}$, the corresponding deterministic function-oracle cost has accuracy power
\[
\varepsilon^{-(1-\nu_i)/(1+3\nu_i)},
\]
with other parameters fixed, while the coupling cost remains logarithmic. These distance guarantees should be distinguished from the uniform gap bound above.

Higher derivatives now give stronger pointwise guarantees as well. For $0\in F(z)+G(z)$, let $F:\R^n\to\R^n$ be monotone, $G$ maximally monotone, and suppose a solution $z^*$ exists. For an integer order $j\ge2$, assume that $D^{j-1}F$ is $L_j$-Lipschitz in operator norm. The anchored extra-proximal method of \cite{jiang2026anchored} attains tangent residual
\[
r_{\mathrm{tan}}(z):=\inf_{v\in G(z)}\|F(z)+v\|_2\le\varepsilon
\]
with derivative-oracle complexity
\[
\widetilde{\cO}_j\left(1+
\left(\frac{L_jR^j}{\varepsilon}\right)^{2/(3j-1)}\right),
\qquad R=\|z^0-z^*\|_2.
\]
One query at $z$ returns $F(z),DF(z),\ldots,D^{j-1}F(z)$; regularized Taylor-model inclusions are solved exactly, and their arithmetic cost is not included. The logarithmic overhead comes from line search. A matching lower bound in sufficiently high dimension covers every deterministic algorithm in this oracle model. In particular, the second-order accuracy exponent is $2/5$. On a bounded saddle domain of diameter $D$, the duality gap is at most $D r_{\mathrm{tan}}(z)$, so the pointwise result also improves the gap guarantee.

Independently, \cite{zhang2026matchingvi} obtains the same higher-order exponent for monotone VIs on a known compact convex domain, with exact derivative queries and Taylor-model solves. Its matching lower bound also covers randomized algorithms with a fixed query budget and per-instance success probability at least $2/3$. These results establish the optimal accuracy exponent in the stated derivative-oracle models; they leave the costs of model solution, stochastic derivatives, and distributed implementation as separate questions.

\subsection{Wasserstein barycenter example}
The Wasserstein barycenter problem provides a useful example in which the statistical formulation and the computational representation both matter. A population barycenter minimizes an expected transportation cost over probability measures; its motivation is statistical but need not come from maximum likelihood \cite{boissard2015distribution,bigot2018characterization}. The empirical problem replaces this expectation by an average of observed measures.

For discrete measures on a fixed support, entropy regularization leads to useful saddle-point and dual representations \cite{dvinskikh2021decentralized}. In the dense finite-support model with $n$ support points considered there, a dual gradient can be evaluated in order $n^2$ arithmetic operations, and a sampled unbiased estimate can cost order $n$. These costs are tied to the regularized representation, the cost matrix, and the chosen oracle; they are not universal bounds for arbitrary transportation problems. Direct primal evaluations may themselves require a transport solve, whose cost depends on the requested accuracy and regularization. This difference can make the dual representation computationally attractive.

The statistical use of that representation requires an error budget for empirical approximation, regularization bias, and numerical optimization \cite{dvinskikh2021stochastic}. Entropic geometry is natural on the simplex and may improve dimension dependence, but Euclidean regularization is also possible. Moreover, entropy-related gradients can become unbounded near the boundary. Applying a uniform Lipschitz theorem may therefore require the restricted simplex
\[
\Delta_n^\rho=\{p\in\R^n:p_i\ge\rho,\ \textstyle\sum_{i=1}^n p_i=1\},
\qquad 0<\rho<1/n,
\]
or a separate local analysis. The approximation introduced by such a restriction must also be controlled.

Strong convexity introduced by an appropriate regularization can yield a smooth conjugate and hence a smooth dual problem \cite{rockafellar1970convex}. This supports accelerated distributed methods for the empirical barycenter. For a non-smooth primal formulation, averaging stochastic subgradients is still possible and reduces their variance; what does not follow automatically is the smooth accelerated bound on sequential iterations. Similarly, the empirical dual construction is an advantage of this particular offline reduction, rather than a prohibition on all online dual or streaming methods.

Recent work makes the balance between computation and statistical bias more explicit. Doubly regularized barycenters use an inner entropy parameter for transportation and an outer entropy parameter for the barycenter \cite{chizat2025doubly}. Positive regularization yields stability and smooth densities. For the quadratic transportation cost, under the stated compactness, smooth-density, and Fisher-information assumptions, choosing the outer parameter as half the inner one improves the bias in the unregularized barycenter objective from first to second order in the inner parameter. The same framework admits globally convergent particle dynamics in a mean-field limit. Thus regularization can improve the approximation-computation tradeoff, rather than only make a finite-dimensional problem easier to solve. Streaming methods directly in Wasserstein space provide a complementary online route under their geometric and regularity assumptions \cite{backhoff2024barycenters}.

The barycenter example brings together the main themes of the survey. A statistically adequate sample size does not determine the best optimization method. Geometry, the representation of the oracle, regularization, and opportunities for parallel computation can be equally decisive. The advances in adaptive, robust, variance-reduced, and higher-order methods enlarge the range of problems for which these choices can be made on a quantitative basis.

\section{Historical Notes}
Stochastic approximation emerged as a distinct subject with the seminal paper of Robbins and Monro \cite{robbins1951stochastic}. Its development has been closely connected with mathematical statistics, control, and deterministic optimization. The book of Nemirovski and Yudin \cite{nemirovski1983problem}, first published in Russian in 1979, established a complexity framework that already included stochastic first-order information. The growth of data science brought another wave of interest, making stochastic optimization a central computational tool. General treatments include \cite{ermoliev1988numerical,birge2011introduction,schneider2007stochastic,prekopa2013stochastic,spall2012stochastic,shapiro2021lectures}; connections with learning and data analysis are developed in \cite{sra2012optimization,sridharan2012learning,shalev2014understanding,duchi2018introductory,bottou2018optimization,lan2020first,bach2021learning,wright2022optimization}.

The statistical ideas in Section~\ref{Sec:StatisticsMotiv} are classical; see \cite{ibragimov2013statistical,spokoiny2015basics}. Finite-sample and misspecified-model perspectives are discussed in \cite{spokoiny2012parametric,spokoiny2013bernstein}. The interpretation of regularization through prior information also connects with Bayesian learning, as illustrated in \cite{krasotkina2015bayesian}. Iterate averaging in stochastic approximation was proposed independently by Ruppert \cite{ruppert1988efficient} and Polyak \cite{polyak1990new}, and analyzed by Polyak and Juditsky \cite{polyak1992acceleration}. Asymptotically efficient online estimation was studied earlier in \cite{polyak1979adaptive,polyak1980optimal}.

Section~\ref{sec:ML} draws on \cite{kakade2008generalization,sridharan2008fast,shalev2009stochastic,shalev2010learnability}. Broader accounts of online learning are given in \cite{cesa2006prediction,hazan2016introduction,orabona2019modern,cesa2021online}. Mirror descent extends Euclidean arguments to other geometries, notably the probability simplex. In the nonsmooth setting a stochastic subgradient must be a measurable selection with the required conditional expectation and norm or moment bound. Lipschitz continuity on a constrained set does not bound every relative subgradient, because normal vectors can occur at its boundary; see \cite{shapiro2021lectures} for the underlying interchange and measurability questions.

The SAA covering argument of Section~\ref{Sec:NonConv_Conv} is classical \cite{shapiro2005complexity,shapiro2021lectures}, whereas the SA bounds follow the mirror-descent theory in \cite{nemirovski2009robust}. Information-theoretic first-order lower bounds are given in \cite{agarwal2012information}. Feldman's work \cite{feldman2016generalization} established dimension-dependent lower bounds for ERM and distinguished it from general learnability. The matching characterization by \cite{pmlr-v238-carmon24a} also separates ERM from uniform convergence; the general covering estimate is therefore not the optimal convex ERM estimate.

The strongly convex SA results in Section~\ref{Sec:SCR} are related to \cite{juditsky2014deterministic}. The corresponding SAA results use stability and regularization \cite{shalev2009stochastic,shalev2014understanding}. High-probability stability bounds are developed in \cite{feldman2019high,klochkov2021stability}; the distinction between these bounds and expected-risk bounds is essential. Tikhonov regularization has a much earlier origin \cite{tikhonov1977solutions}. Non-Euclidean SAA--SA comparisons are studied in \cite{dvinskikh2021stochastic,dvinskikh2022relations}.

In Section~\ref{Sec:rgrwt}, the offline growth-condition analysis belongs to SAA theory \cite{shapiro2005complexity,shapiro2021lectures}. The online analysis uses localization and restarts, including \cite{juditsky2014deterministic} for quadratic growth and related uniform-convexity assumptions. Sharp minima were studied in \cite{poljak1982sharp}, and early stochastic results appear in \cite{juditsky1993stochastic}. Restart ideas go back to \cite{nemirovski1983problem,nemirovski1985optimal}; see also \cite{roulet2017sharpness} for their relation to sharpness and acceleration.

\end{fulltext}

\bibliographystyle{abbrvnat}
\bibliography{Chapter_1/ch1_refs}

\section*{Note on the preparation of the article}
The article was structured and written by the authors in 2020 without the use of artificial intelligence, with minor additions in subsequent years. In 2026, the survey was expanded to include recent results and references with assistance from ChatGPT. Preparation of the resulting text involved substantial use of ChatGPT for corrections to the English and for checking and updating the bibliography.

\end{document}